\documentclass[12pt]{amsart}

\usepackage{amssymb,latexsym}

\usepackage{enumerate}

\makeatletter

\@namedef{subjclassname@2010}{

  \textup{2010} Mathematics Subject Classification}

\makeatother
\newtheorem{thm}{Theorem}[section]
\newtheorem*{thma}{Theorem A}
\newtheorem*{thmb}{Theorem B}
\newtheorem*{thmc}{Theorem C}
\newtheorem{Thm}{Theorem}
\newtheorem{Con}{Conjecture}
\newtheorem{Cor}{Corollary}

\newtheorem{cor}[thm]{Corollary}
\newtheorem{lem}[thm]{Lemma}
\newtheorem{pro}[thm]{Proposition}
\theoremstyle{definition}
\newtheorem{defin}{Definition}

\numberwithin{equation}{section}

\begin{document}

\baselineskip=17pt

\title{Prime number races with three or more competitors}

\author[Youness Lamzouri]{Youness Lamzouri}

\address{Department of Mathematics, University of Illinois at Urbana-Champaign,
1409 W. Green Street,
Urbana, IL, 61821
USA}

\email{lamzouri@math.uiuc.edu}

\date{}

\begin{abstract} Fix an integer $r\geq 3$. Let $q$ be a large positive integer and $a_1,\dots,a_r$ be distinct residue classes modulo $q$  that are relatively prime to $q$. In this paper, we establish an asymptotic formula for the logarithmic density $\delta_{q;a_1,\dots,a_r}$ of the set of real numbers $x$ such that $\pi(x;q,a_1)>\pi(x;q,a_2)>\dots>\pi(x;q,a_r),$ as $q\to\infty$; conditionally on the assumption of the Generalized Riemann Hypothesis GRH and the Grand Simplicity Hypothesis GSH. Several applications concerning these prime number races are then deduced. Indeed, comparing with a recent work of D. Fiorilli and G. Martin \cite{FiM}  for the case $r=2$, we show that these densities behave differently when $r\geq 3$. Another consequence of our results is the fact that, unlike two-way races,  biases do appear in races involving three of more squares (or non-squares) to large moduli. Furthermore, we establish a conjecture of M. Rubinstein and P. Sarnak \cite{RS} (on biased races)  in certain cases where the $a_i$ are assumed to be fixed and $q$ is large. We also prove that a  conjecture of A. Feuerverger and G. Martin \cite{FeM} concerning ``bias factors'' (which follows from the work of Rubinstein and Sarnak \cite{RS} for $r=2$) does not hold when $r\geq 3$. Finally, we use a variant of our method to derive Fiorilli and Martin \cite{FiM} asymptotic formula for the densities in two-way races.

\end{abstract}



\thanks{The author is supported by a postdoctoral fellowship from the Natural Sciences and Engineering Research Council of Canada.}

\maketitle
\tableofcontents

\section{Introduction}

\noindent In 1853 Chebyshev observed that primes congruent to $3$ modulo $4$ seem to predominate over those congruent to $1$ modulo $4$. In general, if $a$ is a non-square modulo $q$ and $b$ is a square modulo $q$ then $\pi(x;q,a)$ has a strong tendency to be larger than $\pi(x;q,b)$, where $\pi(x;q,a)$ denotes the number of primes less than $x$ that are congruent to $a$ modulo $q$. This general phenomenon is known as ``Chebyshev's bias''. This bias might appear unexpected in view of the prime number theorem for arithmetic progressions which states that  $\lim_{x\to\infty}\pi(x;q,a)/\pi(x;q,b)=1, \text{ as } x\to \infty$, for any $a$ and $b$ that are coprime to $q$.
 In fact, this asymptotic result does not give us any information on the difference $\pi(x;q,a)-\pi(x;q,b)$. In 1914, J.E. Littlewood \cite{Li} proved that the quantities $\pi(x;4,3)-\pi(x;4,1)$ and $\pi(x;3,2)-\pi(x;3,1)$ change sign infinitely often. Similar results to other moduli were subsequently derived by S. Knapowski and P. Tur\'an \cite{KT} (under some hypotheses on the zeros of Dirichlet $L$-functions), and further generalizations of this question were considered by J. Kaczorowski \cite{Ka1}, \cite{Ka2}.

Chebyshev's observation was the origin for a big branch of modern Number Theory, namely, comparative prime number theory. For a complete history of this subject, one can refer to the delightful articles of A. Granville and G. Martin \cite{GM}, and K. Ford and S. Konyagin \cite{FK2}.

A generalization of Chebyshev's question is the so called `` Shanks and R\'enyi prime number races problem'' which is described in the following way. Let $q\geq 3$ and $2\leq r\leq \phi(q)$ be positive integers. Define $\mathcal{A}_r(q)$ to be the set of ordered $r$-tuples of distinct residue classes $(a_1,a_2,\dots, a_r)$ modulo $q$ which are coprime to $q$. For $(a_1,a_2,\dots,a_r)\in \mathcal{A}_r(q)$, let $P_{q; a_1,\dots, a_r}$ be the set of real numbers $x\geq 2$ such that
$$ \pi(x;q,a_1)>\pi(x;q,a_2)>\dots>\pi(x;q,a_r).$$
Will the sets $P_{q; a_1,\dots, a_r}$ contain arbitrarily large  values, for every $r$-tuple $(a_1,a_2,\dots,a_r)\in \mathcal{A}_r(q)$?
In their fundamental work of 1994, M. Rubinstein and P. Sarnak \cite{RS} solved this question assuming the Generalized Riemann Hypothesis GRH and the Grand Simplicity Hypothesis GSH (which is the assumption that the imaginary parts of the zeros of all Dirichlet $L$-functions attached to primitive characters modulo $q$ are linearly independent over $\mathbb{Q}$). Indeed, they showed that for any $(a_1,\dots, a_r)\in \mathcal{A}_r(q)$ the logarithmic density of $P_{q;a_1,\dots, a_r}$ defined by
$$ \delta_{q; a_1,\dots,a_r}:= \lim_{x\to \infty}\frac{1}{\log x}\int_{t\in P_{q;a_1,\dots, a_r}\cap [2,x]}\frac{dt}{t},$$
exists and is $>0$. In fact this is corollary of a stronger result they proved, that there exists an absolutely continuous measure (with respect to the Lebesgue measure on $\mathbb{R}^r$) $\mu_{q;a_1,\dots,a_r}$ such that
\begin{equation}
 \delta_{q;a_1,\dots,a_r} =\int_{\substack{(x_1,\dots,x_r)\in \mathbb{R}^r\\ x_1>x_2>\cdots >x_r}}d\mu_{q;a_1,\dots,a_r}(x_1,\dots,x_r).
\end{equation}

All the results we obtain in this paper are conditional on the same two hypotheses (namely GRH and GSH) as the work of Rubinstein and Sarnak. In \cite{FK1}, Ford and Konyagin showed that assumptions on the locations of the zeros of Dirichlet $L$-functions are indeed necessary in order to obtain results on prime number races with three or more competitors.

In the case of a race between two residue classes $a$ and $b$ modulo $q$, Rubinstein and Sarnak proved that $\delta_{q;a,b}=\delta_{q;b,a}=1/2$ if $a$ and $b$ are both squares or both non-squares modulo $q$, and otherwise $\delta_{q;a,b}>1/2$ if $a$ is a non-square and $b$ is a square modulo $q$ (note that $\delta_{q;b,a}=1-\delta_{q;a,b}$). They also showed that $\delta_{q;a,b}\to 1/2$ as $q\to \infty$, uniformly for all distinct reduced residue classes $a,b$ modulo $q$. In fact, they proved that in general all biases disappear when $q\to\infty$. Let
$$ \Delta_r(q):= \max_{(a_1,a_2,\dots,a_r)\in \mathcal{A}_r(q)}\left|\delta_{q;a_1,\dots,a_r}-\frac{1}{r!}\right|.$$
 Then for any fixed $r\geq 2$, Rubinstein and Sarnak showed that assuming GRH and GSH, we have
\begin{equation}
\Delta_r(q)\to 0 \text{ as } q\to \infty.
\end{equation}
For $r=2$, D. Fiorilli and G. Martin \cite{FiM} have recently established an asymptotic expansion for $\delta_{q;a,b}-1/2$ when $a$ is a non-square and $b$ is a square modulo $q$. A corollary of their results is that for $q$ large

$$\Delta_2(q)=\frac{1}{q^{1/2+o(1)}}.$$
A surprising consequence of our results is that
$\Delta_r(q)$ behaves in a complete different way when $r\geq 3$.
\begin{thma} Assume GRH and GSH. Let $r\geq 3$ be a fixed integer. If $q$ is large, we have
$$\Delta_r(q)\asymp_r \frac{1}{\log q}.$$
\end{thma}

Recall that a bias occurs in a two-way race $\{q;a_1,a_2\}$ if and only if  one of  the residue classes $a_1$ and $a_2$ is a square and the other is a non-square modulo $q$. An interesting problem is then to determine when these biases appear for general races $\{q;a_1,\dots,a_r\}$ with $r\geq 3$. To make things clear we need to precisely define the notions of ``biased'' and ``unbiased'' races.  Although Rubinstein and Sarnak called a race $\{q;a_1,\dots,a_r\}$  unbiased if the density function associated to the measure $\mu_{q;a_1,\dots,a_r}$ is symmetric, we believe that a more appropriate definition is the following
\begin{defin} Let $(a_1,\dots,a_r)\in \mathcal{A}_r(q)$. The race $\{q;a_1,\dots,a_r\}$ is said to be \emph{unbiased} if for every permutation $\sigma$ of the set $\{1,2,\dots,r\}$ we have
$$ \delta_{q;a_{\sigma(1)},\dots, a_{\sigma(r)}}=\delta_{q;a_1,\dots,a_r}=\frac{1}{r!}.$$
Furthermore, a race is said to be \emph{biased} if this condition does not hold.
\end{defin}
While investigating these biases we made the following interesting observation that if the race $\{q;a_1,\dots,a_r\}$ is unbiased then the races $\{q;a_{i_1},\dots,a_{i_s}\}$ are unbiased for any subset $\{i_1,\dots,i_s\}$ of $\{1,\dots, r\}.$ In view of  Rubinstein and Sarnak results on two-way races, this clearly shows that a race $\{q; a_1,\dots,a_r\}$ is biased if there are $1\leq i \neq j \leq r$ such that $a_i$ is a square and $a_j$ is a non-square modulo $q$. Furthermore, it is obvious from (1.1) that the race $\{q;a_1,\dots,a_r\}$ is unbiased if the density function of $\mu_{q;a_1,\dots,a_r}$ is symmetric. Rubinstein and Sarnak investigated the Fourier transform of $\mu_{q;a_1,\dots,a_r}$ for $r\geq 3$, and showed that the only case when this distribution is symmetric occurs when $r=3$ and
\begin{equation}
 a_2\equiv a_1 \rho \text{ mod } q, a_3\equiv a_1 \rho^2 \text{ mod }q,
\end{equation}
for some $\rho \neq 1$ with $\rho^3\equiv 1 \text{ mod } q.$ However, this result still leaves
open the possibility that unbiased races not verifying assumption (1.3) might exist (since, for example, a function can be positive half of the time without being symmetric). Nonetheless, Rubinstein and Sarnak conjectured that the only case when a race involving three or more competitors is unbiased corresponds to (1.3).
\begin{Con}[Rubinstein and Sarnak \cite{RS}] When $r\geq 3$, the race $\{q;a_1,\dots,a_r\}$ is unbiased if and only if $r=3$ and the residue classes $a_1, a_2,$ and $a_3$ satisfy assumption (1.3).
\end{Con}
A. Feuerverger and G. Martin \cite{FeM} were the first to exhibit explicit examples of biased races with three competitors, where the residue classes are either squares or non-squares not satisfying assumption (1.3). For example, they showed that the races $\{8;3,5,7\}$ and $\{12;5,7,11\}$ are biased. However, all the examples they considered satisfy $r\leq 4$ and $q\leq 12$, thus leaving open the problem of determining the existence of biased races of this type for any $q> 12 $ and $3\leq r\leq \phi(q).$ We solved this question for any fixed $r\geq 3$ if $q$ is large enough. Indeed we show that unlike two-way races, biases do appear in races involving three of more squares (or non-squares) modulo $q$, if $q$ is large.
\begin{thmb} Assume GRH and GSH. Let $r\geq 3$ be a fixed integer. Then there exists a positive number $q_0(r)$ such that for any $q\geq q_0(r)$ there are residue classes $(a_1,\dots,a_r), (b_1,\dots, b_r)\in \mathcal{A}_r(q)$, with $a_1,\dots,a_r$ being all squares and $b_1,\dots,b_r$ being all non-squares, such that both the races $\{q;a_1,\dots,a_r\}$ and $\{q;b_1,\dots,b_r\}$ are biased.
\end{thmb}
For distinct non-zero integers  $a_1,\dots, a_r$, we define  $\mathcal{Q}_{a_1,\dots,a_r}$ to be the set of positive integers $q$ such that $a_1,\dots,a_r$ are distinct modulo $q$, and  $(q,a_i)=1$ for all $1\leq i\leq r$.
 When $r=3$, assumption (1.3) implies that $a_1^2\equiv a_2a_3\text{ mod }q$, $a_2^2\equiv a_1a_3\text{ mod }q$, and $a_3^2\equiv a_1a_2\text{ mod }q$. Hence if $q> 2\max(|a_i|^2)$ then these congruences become identities. However, since the $a_i$ are assumed to be distinct these equalities can not hold. This leads to a weak form of Conjecture 1 of Rubinstein and Sarnak:
\begin{Con} Let $r\geq 3$ and $a_1,\dots,a_r$ be distinct non-zero integers. Then for all positive integers $q\in \mathcal{Q}_{a_1,\dots,a_r}$ such that $q>2\max(|a_i|^2)$, the race $\{q;a_1,\dots,a_r\}$ is biased.
\end{Con}

We prove the following partial result towards this conjecture, which follows from Theorem 3 below.
\begin{thmc} Let $r\geq 3$ and $a_1,\dots,a_r$ be distinct non-zero integers such that one of the conditions below occur

i) There exist $1\leq j\neq k\leq r$ such that $a_j+a_k=0$.

ii) There exist $1\leq j\neq k\leq r$ such that $a_j/a_k$ is a prime power.

\noindent Then for all but finitely many $q\in \mathcal{Q}_{a_1,\dots,a_r}$, the race $\{q;a_1,\dots,a_r\}$ is biased.
\end{thmc}
To establish these results, we prove an asymptotic formula for $\delta_{q;a_1,\dots,a_r}$ valid for large $q$, and then we investigate the behavior of its first few terms. Our approach is different from the one used by Fiorilli and Martin \cite{FiM} in the case $r=2$. Indeed their idea consists of reducing the study of the measure $\mu_{q;a_1,a_2}$ (which is a measure on $\mathbb{R}^2$) to a related one-dimensional measure $\rho_q$ on $\mathbb{R}$, using an explicit formula of Feuerverger and Martin \cite{FeM}. Although this approach is natural for $r=2$, it is hardly generalizable to $r\geq 3$, due to the lack of symmetry in this case. Instead, we exploit the fact, used by Rubinstein and Sarnak to prove (1.2), that the Fourier transform of $\mu_{q;a_1,\dots,a_r}$ approaches a multivariate Gaussian in a certain range, when $q\to\infty$.

 In the next section we shall discuss these results in details. In particular we shall describe the asymptotic formula we prove for the densities $\delta_{q;a_1,\dots,a_r}$ and deduce further consequences.

 {\bf Acknowledgments.} I would like to thank Andrew Granville for introducing me to this delightful subject and for many comments and suggestions. I also thank Kevin Ford for several valuable discussions on the results of this paper.
\section{Detailed statement of results}
We shall use the following normalization for the Fourier transform of an integrable function $f:\mathbb{R}^n\to \mathbb{C}$

$$\hat{f}(t_1,\dots , t_n)=\int_{\mathbb{R}^n}e^{-i(t_1x_1+\cdots+ t_nx_n)}f(x_1,\dots, x_n)dx_1 \dots dx_n.$$
Then if $\hat{f}$ is integrable on $\mathbb{R}^n$ we have the Fourier inversion formula
$$ f(x_1,\dots, x_n)=(2\pi)^{-n}\int_{\mathbb{R}^n}e^{i(t_1x_1+\cdots+ t_nx_n)}\hat{f}(t_1,\dots, t_n)dt_1 \dots dt_n.$$
Similarly we write
$$\hat{\nu}(t_1,\dots, t_n)=\int_{\mathbb{R}^n}e^{-i(t_1x_1+\cdots+ t_nx_n)}d\nu(x_1, \dots, x_n)$$
 for the Fourier transform of a finite measure $\nu$ on $\mathbb{R}^n$. For $t\in \mathbb{R}^n$ we shall use the notations $||t||$ and $|t|_{\infty}$ for the Euclidian norm and the maximum norm of $t$ respectively.

 Assuming GRH and GSH, Rubinstein and Sarnak obtained an explicit formula for the Fourier transform of $\mu_{q;a_1,\dots,a_r}$ in terms of the non-trivial zeros of Dirichlet $L$-functions attached to non-principal characters modulo $q$. More specifically they showed that
\begin{equation}
\hat{\mu}_{q;a_1,\dots, a_r}(t_1,\dots,t_r)=  \exp\left(i\sum_{j=1}^rC_q(a_j)t_j\right)\prod_{\substack{\chi\neq \chi_0\\ \chi\text{ mod } q}}\prod_{\gamma_{\chi}>0}J_0\left(\frac{2\left|\sum_{j=1}^r\chi(a_j)t_j\right|}
{\sqrt{\frac14+\gamma_{\chi}^2}}\right)
\end{equation}
for $(t_1,\dots,t_r)\in \mathbb{R}^r$, where $\chi_0$ is the principal character modulo $q$,
$$ C_q(a):=-1+ \sum_{\substack{b^2\equiv a \text{ mod } q\\ 1\leq b\leq q}}1,$$  $J_0(z)=\sum_{m=0}^{\infty}(-1)^{m}(z/2)^{2m}/m!^2$ is the Bessel function of order $0$, and $\{\gamma_{\chi}\}$ denotes the set of imaginary parts of the non-trivial zeros of $L(s,\chi)$. Note that for $(a,q)=1$ the function $C_q(a)$ takes only two values: $C_q(a)=-1$ if $a$ is a non-square modulo $q$, and $C_q(a)=C_q(1)$ if $a$ is a square modulo $q$. An exercise in elementary number theory shows that $C_q(1)\asymp 2^{\omega(q)}$, where $\omega(q)$ denotes the number of distinct prime factors of $q$. In particular this implies that $C_q(a)\ll_{\epsilon}q^{\epsilon}$ for any $\epsilon>0$.

For $r<\phi(q)$ Rubinstein and Sarnak showed that $\hat{\mu}_{q; a_1,\dots,a_r}(t)$ is rapidly decreasing as $||t||\to\infty$ (we shall quantify this statement in Section 3 below) from which they deduced that the measure $\mu_{q; a_1,\dots,a_r}$ is absolutely continuous. Feuerverger and Martin \cite{FeM} obtained a general formula for $\delta_{q; a_1,\dots, a_r}$ in terms of certain variants of the Fourier transform $\hat{\mu}_{q; a_1,\dots, a_r}$, and used these formulas to rigourously compute certain densities for $r\leq 4$ and $q\leq 12$.

In \cite{FiM}, Fiorilli and Martin used Feuerverger and Martin formula for the case $r=2$ to prove an asymptotic formula for the density $\delta_{q;a_1,a_2}$. More precisely they showed that
\begin{equation}
\delta_{q;a_1,a_2}= \frac{1}{2}- \frac{C_q(a_1)-C_q(a_2)}{\sqrt{2\pi V_q(a_1,a_2)}}+ O\left(\frac{C_q(1)^{3}}{ V_q(a_1,a_2)^{3/2}}\right),
\end{equation}
where $V_q(a_1,a_2)= 2N_q-2B_q(a_1,a_2)$ (see (2.3) below).  In Section 8 we shall derive this asymptotic using a slight modification of our method.

Before stating our main result, let us define some notation which shall be used throughout this paper. Let
\begin{equation}
N_q:=2\sum_{\substack{\chi\neq \chi_0\\ \chi\text{ mod } q}}\sum_{\gamma_{\chi}>0}\frac{1}{\frac14+\gamma_{\chi}^2},   \text{ and } B_q(a,b):=\sum_{\substack{\chi\neq \chi_0 \\ \chi\text{ mod } q}}\sum_{\gamma_{\chi}>0}\frac{\chi\left(\frac{b}{a}\right)+\chi\left(\frac{a}{b}\right)}{\frac14 +\gamma_{\chi}^2},
\end{equation}
for $(a,b)\in \mathcal{A}_2(q)$ (recall that $\mathcal{A}_2(q)$ is the set of ordered pairs of distinct reduced residue classes modulo $q$).
It follows from the work of Rubinstein and Sarnak that $N_q\sim\phi(q)\log q.$
Moreover, we shall prove using the work of Fiorilli and Martin that $B_q(a,b)\ll \phi(q).$
We also put
$$ C_q=C_q(a_1,\dots,a_r):=\max_{1\leq j\leq q} |C_q(a_j)|, \text{ and } B_q=B_q(a_1,\dots,a_r):=\max_{1\leq j<k\leq r}|B_q(a_j,a_k)|.$$
Finally for $1\leq j\neq k\leq r$, we define the following integrals which shall appear in the asymptotic of $\delta_{q;a_1,\dots,a_r}$
$$ \alpha_j(r):=(2\pi)^{-r/2}\int_{x_1>x_2>\dots>x_r}x_j\exp\left(-\frac{x_1^2+\cdots+x_r^2}{2}\right) dx_1\dots dx_r,$$
$$ \lambda_j(r):=(2\pi)^{-r/2}\int_{x_1>x_2>\dots>x_r}(x_j^2-1)\exp\left(-\frac{x_1^2+\cdots+x_r^2}{2}\right) dx_1\dots dx_r,$$
and $$\beta_{j,k}(r):=(2\pi)^{-r/2}\int_{x_1>x_2>\dots>x_r}x_jx_k\exp\left(-\frac{x_1^2+\cdots+x_r^2}{2}\right) dx_1\dots dx_r.$$
\begin{Thm} Assume GRH and GSH. Fix an integer $r\geq 2$. If $q$ is a large positive integer and $(a_1, \dots, a_r)\in \mathcal{A}_r(q)$, then
\begin{equation*}
\begin{aligned}
&\delta_{q;a_1,\dots, a_r}=\frac{1}{r!}-\frac{1}{\sqrt{N_q}}\sum_{1\leq j\leq r}\alpha_j(r)C_q(a_j)+\frac{1}{N_q}\sum_{1\leq j<k\leq r}\beta_{j,k}(r)B_q(a_j,a_k)\\
&+\frac{1}{2N_q}\left(\sum_{1\leq j\leq r}\lambda_j(r)C_q(a_j)^2+2\sum_{1\leq j<k\leq r}\beta_{j,k}(r)C_q(a_j)C_q(a_k)\right)+O_r\left(\frac{1}{N_q}+ \frac{C_qB_q}{N_q^{3/2}}+ \frac{B_q^2}{N_q^2}\right).\\
\end{aligned}
\end{equation*}
 \end{Thm}
As a corollary we obtain
 \begin{Cor} Under the same assumptions of Theorem 1 we have
 $$\delta_{q;a_1,\dots, a_r}=\frac{1}{r!}-\frac{1}{\sqrt{N_q}}\sum_{1\leq j\leq r}\alpha_j(r)C_q(a_j)+\frac{1}{N_q}\sum_{1\leq j<k\leq r}\beta_{j,k}(r)B_q(a_j,a_k)+ O_r\left(\frac{C_q^2}{N_q}+ \frac{B_q^2}{N_q^2}\right).$$
 \end{Cor}
 In particular, we get for $r=3$ that
\begin{Cor} Under the same assumptions of Theorem 1 we have
 \begin{equation*}
 \begin{aligned}
 \delta_{q;a_1,a_2, a_3}&=\frac{1}{6}+\frac{1}{4\sqrt{\pi N_q}}(C_q(a_3)-C_q(a_1))\\
 &+\frac{1}{4\pi\sqrt{3}N_q}(
 B_q(a_1,a_2)+B_q(a_2,a_3)-2B_q(a_1,a_3))
 + O\left(\frac{C_q^2}{N_q}+ \frac{B_q^2}{N_q^2}\right).
 \end{aligned}
 \end{equation*}
 \end{Cor}
 \emph{Remark 1}. The main difference between the cases $r=2$ and $r\geq 3$ lies in the fact that $\beta_{1,2}(2)=0$, which implies that the terms involving $B_q(a_j,a_k)$ are missing in the case $r=2$. Indeed, we shall later prove that the contribution of these terms can be $\gg_r 1/\log q$. This explains the surprising behavior of $\Delta_r(q)$ when $r\geq 3$, since $C_q(a)/\sqrt{N_q}=q^{-1/2+o(1)}.$ Remark also that our asymptotic formula is not accurate in the case $r=2$ since the error term may exceed the main term. We shall slightly modify the argument of the proof to handle this case in Section 8.

 Investigating the terms $B_q(a_j,a_k)$ and using the fact that $B_q\ll \phi(q)$, we prove the following result, which is stronger form of Theorem A.
 \begin{Thm} Assume GRH and GSH. Fix an integer $r\geq 3$, and let $q$ be a large positive integer. Then for all $(a_1,\dots,a_r)\in \mathcal{A}_r(q)$ we have
$$\left|\delta_{q;a_1,\dots,a_r}-\frac{1}{r!}\right|\ll_r \frac{1}{\log q}.$$
Moreover there exist residue classes $(b_1,\dots,b_r)$, $(d_1,\dots, d_r)\in \mathcal{A}_r(q)$ such that
$$ \delta_{q;b_1,\dots,b_r}> \frac{1}{r!}+\frac{c_1(r)}{\log q} \ \  \text{ and }  \ \ \delta_{q;d_1,\dots,d_r}<\frac{1}{r!}-\frac{c_1(r)}{\log q},$$
for some constant $c_1(r)>0$ which depends only on $r$.
\end{Thm}
This result implies that for some residue classes $a_1,\dots,a_r$ modulo $q$ the distance $|\delta_{q; a_1,\dots,a_r}-1/r!|$  can be $\gg_r 1/\log q$. An interesting question is then to investigate for which residue classes modulo $q$ does this extreme bias occur. To this end let us make the following definition
\begin{defin} Fix $r\geq 3$ and let $q$ be a large positive integer. We call a race $\{q; a_1,\dots, a_r\}$ ``\emph{$q$-extremely biased}'' if for some permutation $\sigma$ of the set $\{1,\dots,r\}$ we have
$$\left|\delta_{q; a_{\sigma(1)}, \dots, a_{\sigma(r)}}-\frac{1}{r!}\right|\gg_{r}\frac{1}{\log q}.$$
\end{defin}
We can completely characterize $q$-extremely biased races $\{q; a_1,\dots, a_r\}$ when the residue classes $a_1,\dots,a_r$ are bounded and $q$ is large.
\begin{Thm} Assume GRH and GSH. Fix an integer $r\geq 3$ and let $A\geq 1$ be a real number. Then if $a_1,\dots,a_r$ are distinct integers with $|a_i|\leq A$, and $q$ is a large positive integer with $(q,a_i)=1$, the race $\{q; a_1,\dots, a_r\}$ is $q$-extremely biased if and only if one the following conditions occur

i) There exist $1\leq j\neq k\leq r$ such that $a_j+a_k=0$.

ii) There exist $1\leq j\neq k\leq r$ such that $a_j/a_k$ is a prime power.

\noindent Moreover, if neither i) nor ii) hold then for any permutation $\sigma$ of the set $\{1,\dots,r\}$
$$ \left|\delta_{q; a_{\sigma(1)},\dots, a_{\sigma(r)}}-\frac{1}{r!}\right|=
\begin{cases} O_{A,r}\left(\displaystyle{\frac{\log q}{q}}\right) & \text{ if the } a_i \text{ are all squares (or non-squares) mod } q,\\
 O_{\epsilon,r}\left(q^{-1/2+\epsilon}\right) & \text{ otherwise}.\end{cases}$$
\end{Thm}

Since the functions $\sum_{j=1}^rx_j$, $\sum_{j=1}^r(x_j^2-1)$ and $\sum_{1\leq j<k\leq r}x_jx_k$ are symmetric in the variables $x_1,\dots,x_r$ and that $\int_{\mathbb{R}}x\exp(-x^2/2)dx= \int_{\mathbb{R}}(x^2-1)\exp(-x^2/2)dx=0$, we deduce that
\begin{equation}
\sum_{j=1}^r\alpha_j(r)=\sum_{j=1}^r\lambda_j(r)=\sum_{1\leq j<k\leq r}\beta_{j,k}(r)=0.
\end{equation}
Therefore, in the case where the $a_i$ are all squares or all non-squares modulo $q$ we obtain the following corollary of Theorem 1
\begin{Cor} Assume GRH and GSH.  Fix an integer $r\geq 3$, and let $q$ be a large positive integer. Then, for any $(a_1, \dots, a_r)\in \mathcal{A}_r(q)$ such that the $a_i$ are all squares or all non-squares modulo $q$, we have
$$\delta_{q;a_1,\dots, a_r}= \frac{1}{r!}+ \frac{1}{N_q}\sum_{1\leq j<k\leq r}\beta_{j,k}(r)B_q(a_j,a_k)+O_r\left(\frac{1}{N_q}+ \frac{C_qB_q}{N_q^{3/2}}+ \frac{B_q^2}{N_q^2}\right).$$
\end{Cor}
Using this result along with an explicit construction of the residue classes $a_1,\dots,a_r$ modulo $q$, we prove a strong from of Theorem B.
\begin{Thm} Assume GRH and GSH. Fix an integer $r\geq 3$, and let $q$ be a large positive integer. Then there exist residue classes $(a_1,\dots,a_r), (b_1,\dots, b_r)\in \mathcal{A}_r(q)$, with $a_1,\dots,a_r$ being all squares and $b_1,\dots,b_r$ being all non-squares modulo $q$, and a permutation $\sigma$ of the set $\{1,\dots, r\}$, such that
$$ \delta_{q; a_1,\dots, a_r}= \delta_{q; b_1, \dots, b_r}< \frac{1}{r!} -\frac{c_2(r)}{\log^3 q} \text{ and } \delta_{q;
 a_{\sigma(1)},\dots, a_{\sigma(r)}}= \delta_{q;  b_{\sigma(1)}, \dots, b_{\sigma(r)}}> \frac{1}{r!} +\frac{c_2(r)}{\log^3 q},$$
for some constant $c_2(r)>0$ which depends only on $r$.
\end{Thm}

 \emph{Remark 2.} If $-1$ is a square modulo $q$, or $(q,p_1p_2)=1$ for some fixed primes $p_1\neq p_2$, then we can replace $c_2(r)/\log^3 q$ by $c_2(r)/\log q$ in the statement of Theorem 4.

It is clear from Theorem 1 that in order to understand the behavior of $\delta_{q;a_1,\dots,a_r}$, we have to investigate the size of $B_q(a,b)$ for $(a,b)\in \mathcal{A}_2(q).$ Recall that $B_q(a,b)\ll \phi(q).$ On the other hand we shall prove that this bound is attained if $a+b\equiv 0\text{ mod } q$, so that
 $\max_{(a,b)\in \mathcal{A}_2(q)}|B_q(a,b)|\asymp \phi(q)$.
 An interesting question is then to determine the order of magnitude of $|B_q(a,b)|$ for a generic pair $(a,b)\in \mathcal{A}_2(q)$. We prove that on average $|B_q(a,b)|\asymp\log q$.
\begin{Thm} Assume GRH. Let $q$ be a large positive integer. Then
$$\log q + O(\log\log q) \leq \frac{1}{|\mathcal{A}_2(q)|}\sum_{(a,b)\in \mathcal{A}_2(q)}|B_q(a,b)|\leq 10 \log q + O(\log\log q).$$
\end{Thm}

In trying to quantify the biases for $r$-tuples $(a_1,\dots,a_r)\in \mathcal{A}_r(q),$ Feuerverger and Martin \cite{FeM} conjectured that there should exist a ``bias factor'' $F_q(a_1,\dots,a_r)$, defined as a linear combination of the $C_q(a_j)$ such that
\begin{equation}
 F_q(a_1,\dots,a_r)>F_q(b_1,\dots,b_r) \implies \delta_{q;a_1,\dots,a_r} >\delta_{q;b_1,\dots,b_r}.
\end{equation}
This is equivalent to say that the inequality on the RHS of (2.5) can be determined only by knowing whether $a_j$ is a square or a non-square modulo $q$ for all $1\leq j\leq r$.
This conjecture is true for $r=2$, as shown by the work of Rubinstein and Sarnak (in this case take $F_q(a_1,a_2)= C_q(a_2)-C_q(a_1)$). Using an explicit construction which involves Burgess's bound for the least quadratic non-residue modulo a prime (see Chapter 12 of \cite{IK}), we show that this conjecture does not hold for $r\geq 3$ if $q$ is large, for any choice of the bias factor. More precisely we prove
\begin{Thm} Assume GRH and GSH. Let $r\geq 3$ be a fixed integer, and $(\kappa_{1},\dots,\kappa_{r})\in \mathbb{R}^r$ such that $(\kappa_{1},\dots,\kappa_{r})\neq (0,\dots,0)$. If $q$ is a large positive integer, there exist two $r$-tuples $(a_1,\dots,a_r),$ $(b_1,\dots, b_r)\in \mathcal{A}_r(q)$ such that
$$ \sum_{1\leq j\leq r} \kappa_jC_q(a_j)>\sum_{1\leq j\leq r} \kappa_jC_q(b_j) \text{ and } \delta_{q;a_1,\dots,a_r}<\delta_{q;b_1,\dots,b_r}.$$
\end{Thm}

On the other direction, combining Theorem 1 and Theorem 5 we show that this conjecture holds for almost all $r$-tuples $(a_1,\dots,a_r)\in\mathcal{A}_r(q)$, with $F_q(a_1,\dots,a_r)=-\sum_{1\leq j\leq r}\alpha_j(r)C_q(a_j)$.
\begin{Thm} Assume GRH and GSH. Fix an integer $r\geq 3$ and let $q$ be a large positive integer. Then there is a set $\Omega_r(q)\subset\mathcal{A}_r(q)$ with $|\Omega_r(q)|=
o(|\mathcal{A}_r(q)|)$, such that for all $r$-tuples $(a_1,\dots,a_r)$, $(b_1,\dots,b_r)\in \mathcal{A}_r(q)\setminus \Omega_r(q)$ we have
$$ -\sum_{j=1}^r\alpha_j(r)C_q(a_j)> -\sum_{j=1}^r\alpha_j(r)C_q(b_j)\implies \delta_{q;a_1,\dots,a_r}>\delta_{q;b_1,\dots,b_r}.$$
\end{Thm}

The plan of the paper is as follows. In the next section we study properties of the Fourier transform $\hat{\mu}_{q;a_1,\dots,a_r}$. These are then used to derive the asymptotic formula of Theorem 1 which is proved in Section 4. In Section 5 we study the behavior of $B_q(a,b)$ on average and prove Theorems 5 and 7. In Section 6 we describe the signs and extreme of values of $B_q(a,b)$, and use these to explicitly construct biased races and prove Theorems 2, 4 and 6.  In Section 7 we study $q$-extremely biased races and prove Theorem 3. Finally, in Section 8 we derive Fiorilli and Martin asymptotic formula for the densities in two-way races.

\section{The Fourier transform $\hat{\mu}_{q;a_1,\dots,a_r}$}

For a non-trivial character $\chi$ modulo $q$, we let $q^*_{\chi}$ be the conductor of $\chi$, and $\chi^*$ be the unique primitive character modulo $q^*_{\chi}$ which induces $\chi$. First we record some standard formulas.
\begin{lem} Assume GRH. Let $\chi$ be a non-trivial character modulo $q$. Then there exists an absolute constant $\gamma_0$ such that
\begin{equation}
\sum_{\gamma_{\chi}}\frac{1}{\frac14+ \gamma_{\chi}^2}
= \log q_{\chi}^*+2\hbox{Re}\frac{L'(1,\chi^*)}{L(1,\chi^*)} -\chi(-1)\log 2+ \gamma_0.
\end{equation}
Moreover, we have
\begin{equation}
 \sum_{\chi \text{ mod } q}\chi(a)\log q_{\chi}^*=\begin{cases} \displaystyle{\phi(q)\left(\log q-\sum_{p|q}\frac{\log p}{p-1}\right)} & \text{ if } a\equiv 1 \text{ mod } q,\\
 \displaystyle{-\phi(q)\frac{\Lambda(q/(q,a-1))}{\phi(q/(q,a-1))}}& \text{ otherwise.}\\
 \end{cases}
\end{equation}
and
$$ N_q=\phi(q)\log q + O(\phi(q)\log\log q).$$
\end{lem}
\begin{proof}
The classical formula (3.1) can be derived from formulas (17) and (18) of chapter 12 in \cite{Da}. Indeed since GRH is assumed, these formulas imply that
$$ \sum_{\gamma_{\chi}}\frac{1}{\frac14+ \gamma_{\chi}^2}
= \log q_{\chi}^*+2\text{Re}\frac{L'(1,\chi^*)}{L(1,\chi^*)}+\text{Re} \frac{\Gamma'\left(\frac{1}{2}+ \frac{1}{2}a\right)}{\Gamma\left(\frac{1}{2}+ \frac{1}{2}a\right)},$$
where $a=0$ if $\chi(-1)=1$ and $a=1$ if $\chi(-1)=-1$. Then (3.1) follows upon taking $\gamma_0= \Gamma'(1)/\Gamma(1)-\log 2$ and noting that
$$ \Gamma'(1/2)/\Gamma(1/2)=\Gamma'(1)/\Gamma(1)-2\log 2.$$
Formula (3.2) corresponds to Proposition 3.3 of \cite{FiM}.
Furthermore, recall that
$$ N_q=2\sum_{\chi\neq \chi_0}\sum_{\gamma_{\chi}>0}\frac{1}{\frac14 + \gamma_{\chi}^2}=\sum_{\chi\neq \chi_0}\sum_{\gamma_{\chi}}\frac{1}{\frac14 + \gamma_{\chi}^2},$$
since $\sum_{\gamma_{\chi}<0}1/(\frac14+\gamma_{\chi}^2)=\sum_{\gamma_{\overline{\chi}}>0}1/(\frac14+\gamma_{\overline{\chi}}^2)$ which is clear from the relation $\overline{L(\overline{s},\chi)}= L(s,\overline{\chi})$. On the other hand we have that
$$\sum_{p|q}\frac{\log p}{p-1}\leq \sum_{p\leq (\log q)^2}\frac{\log p}{p-1}+\frac{1}{\log q}\sum_{p|q}1\ll \log\log q,$$
using the trivial bound $\sum_{p|q}1 \leq \log q/\log2$. Hence, the asymptotic for $N_q$ follows upon combining this last estimate with formulas (3.1) and (3.2) along with the classical result of Littlewood \cite{Li} that $L'/L(1,\chi^{*})=O(\log\log q)$, under GRH.
\end{proof}

Rubinstein and Sarnak \cite{RS} noted that $\hat{\mu}_{q;a_1,\dots,a_r}(t)$ is rapidly decreasing as $||t||\to \infty$. The following result gives a quantitative statement of this decay. More precisely we establish an exponentially decreasing upper bound for $\hat{\mu}_{q;a_1,\dots,a_r}(t)$ which depends on both $t$ and $q$.

\begin{pro} Assume GRH and GSH. Fix an integer $r\geq 2$. Let $q$ be a large positive integer, and let $0<\epsilon<1/2$ be a real number. Then, uniformly for all $(a_1,\dots,a_r)\in\mathcal{A}_r(q)$ we have
$$ |\hat{\mu}_{q;a_1,\dots, a_r}(t_1,\dots,t_r)|\leq \exp(-c_3(r)\phi(q)||t||),$$ for  $t=(t_1,\dots,t_r)\in \mathbb{R}^r$ with $||t||\geq 400$ and
$$ |\hat{\mu}_{q;a_1,\dots, a_r}(t_1,\dots,t_r)|\leq \exp(-c_4(r)\epsilon^2\phi(q)\log q)$$ for  $\epsilon\leq ||t||\leq 400$, where $c_3(r)$ and $c_4(r)$ are positive constants that depend only on $r$.
\end{pro}
\begin{proof}
We begin by proving the first inequality. For any non-trivial character $\chi\text{ mod } q$ we define
$$ F(x,\chi)=\prod_{\gamma_{\chi}>0}J_0\left(\frac{2x}{\sqrt{\frac14+\gamma_{\chi}^2}}\right).$$ Then the explicit formula (2.1) implies that
$$|\hat{\mu}_{q;a_1,\dots, a_r}(t_1,\dots,t_r)|= \prod_{\substack{\chi\neq \chi_0\\ \chi\text{ mod } q}}\left|F\left(\left|\sum_{j=1}^r\chi(a_j)t_j\right|,\chi\right)\right|.$$ By Lemma 2.16 of \cite{FiM} we know that there exists an absolute constant $c>0$ such that
\begin{equation}
|F(x,\chi)F(x,\overline{\chi})|\leq e^{-cx}
\end{equation}
 for $x\geq 200$. On the other hand note that $|F(x,\chi)|\leq 1$ since $|J_0(x)|\leq 1$.

Let $M_q$ be the set of non-trivial characters $\chi\text{ mod }q$ such that $\left|\sum_{j=1}^r\chi(a_j)t_j\right|\geq ||t||/2$. Remark that $\chi \in M_q$ if and only if $\overline{\chi}\in M_q$. Moreover, if $\chi \in M_q$ and $||t||\geq 400$ then $\left|\sum_{j=1}^r\chi(a_j)t_j\right|\geq 200$, which implies
\begin{equation}
\begin{aligned}
|\hat{\mu}_{q;a_1,\dots, a_r}(t_1,\dots,t_r)|^2&\leq \prod_{\chi\in M_q}\left|F\left(\left|\sum_{j=1}^r\chi(a_j)t_j\right|,\chi\right)\right|^2\\
&= \prod_{\chi\in M_q}\left|F\left(\left|\sum_{j=1}^r\chi(a_j)t_j\right|,\chi\right)F\left(\left|\sum_{j=1}^r\chi(a_j)t_j\right|,\overline{\chi}\right)\right|\\
&\leq \exp\left(-c\sum_{\chi\in M_q}\left|\sum_{j=1}^r\chi(a_j)t_j\right|\right)\leq \exp\left(-\frac{c}{2}|M_q|||t||\right),
\end{aligned}
\end{equation}
using (3.3) along with the fact that every character in $M_q$ appears once as $\chi$ and once as $\overline{\chi}$ in the product on the RHS of (3.4).
Thus it only remains to prove a non-trivial lower bound for $|M_q|$. Let
\begin{equation}
\begin{aligned}
 S(t)&=\sum_{\substack{\chi\neq \chi_0 \\ \chi\text{ mod } q}}\left|\sum_{j=1}^r\chi(a_j)t_j\right|^2=  \sum_{\chi\text{ mod } q}\left|\sum_{j=1}^r\chi(a_j)t_j\right|^2-\left(\sum_{j=1}^rt_j\right)^2\\
&= \sum_{j=1}^r\sum_{k=1}^rt_jt_k\sum_{\chi\text{ mod } q}\chi(a_j)\overline{\chi(a_k)}-\left(\sum_{j=1}^rt_j\right)^2=\phi(q)\sum_{j=1}^rt_j^2-\left(\sum_{j=1}^rt_j\right)^2\\&\geq (\phi(q)-r)\sum_{j=1}^rt_j^2,
\end{aligned}
\end{equation}
which follows from the Cauchy-Shwarz inequality. Therefore using that
$\left|\sum_{j=1}^r\chi(a_j)t_j\right|^2\leq \left(\sum_{j=1}^r|t_j|\right)^2\leq r||t||^2$ we deduce
$$
 S(t)=\sum_{\chi\in M_q}\left|\sum_{j=1}^r\chi(a_j)t_j\right|^2+ \sum_{\chi\notin M_q}\left|\sum_{j=1}^r\chi(a_j)t_j\right|^2
\leq r|M_q|||t||^2+ \frac{\phi(q)}{4}||t||^2.
$$
Hence, combining this estimate with (3.5) we obtain
$|M_q|\geq \phi(q)/(2r)$ if $q$ is large enough. This together with (3.4) yield the first part of the proposition.

Now assume that $\epsilon\leq ||t||\leq 400$.  If $\chi\in M_q$ then $2\left|\sum_{j=1}^r\chi(a_j)t_j\right|\geq ||t||\geq \epsilon$. We also note that $\epsilon\left(\frac14+x^2\right)^{-1/2}\leq 1$, for any $x\in\mathbb{R}$. Hence, since $J_0$ is a positive decreasing function on $[0,1]$ and $|J_0(z)|\leq J_0(1)$ for all $z\geq 1$, we get
$$ |\hat{\mu}_{q;a_1,\dots, a_r}(t_1,\dots,t_r)|\leq \prod_{\chi\in M_q}\prod_{\gamma_{\chi}>0}\left|J_0\left(\frac{2\left|\sum_{j=1}^r\chi(a_j)t_j\right|}{\sqrt{\frac14+\gamma_{\chi}^2}}\right)\right|\leq \prod_{\chi\in M_q}\prod_{\gamma_{\chi}>0}\left|J_0\left(\frac{\epsilon}{\sqrt{\frac14+\gamma_{\chi}^2}}\right)\right|.$$
Thus, using the standard bound $|J_0(x)|\leq \exp(-x^2/4)$ for $|x|\leq 1$, we deduce that
\begin{equation}
 |\hat{\mu}_{q;a_1,\dots, a_r}(t_1,\dots,t_r)|\leq \exp\left(-\frac{\epsilon^2}{4}\sum_{\chi\in M_q}\sum_{ \gamma_{\chi}>0}\frac{1}{\frac14+\gamma_{\chi}^2}\right).
\end{equation}
Since $L'/L(1,\chi^{*})=O(\log\log q)$, then equation (3.1) gives
\begin{equation} \sum_{\chi\in M_q}\sum_{ \gamma_{\chi}>0}\frac{1}{\frac14+\gamma_{\chi}^2}=\frac12  \sum_{\chi\in M_q}\sum_{ \gamma_{\chi}}\frac{1}{\frac14+\gamma_{\chi}^2} =\frac12\sum_{\chi\in M_q}\log q_{\chi}^{*}+ O(\phi(q)\log\log q).
\end{equation}
Noting that $q_{\chi}^{*}\leq q$, and using equation (3.2) we derive
$$ \sum_{\chi\in M_q}\log q_{\chi}^{*}\geq \sum_{\chi\text{ mod }q}\log q_{\chi}^{*}- (\phi(q)-|M_q|)\log q \geq \frac{\phi(q)\log q}{2r}+O(\phi(q)\log\log q).$$
The result then follows upon combining this estimate with equations (3.6) and (3.7).
\end{proof}
Our next result (which is a crucial ingredient to the proof of Theorem 1) shows that $\hat{\mu}_{q;a_1,\dots,a_r}$ can be approximated by a multivariate Gaussian in the range $||t||\ll \phi(q)^{-1/2}.$
\begin{pro} Assume GRH and GSH. Fix an integer $r\geq 2$. Then, for any constant $A=A(r)>0$ there exists $L(A)>0$ such that for $L\geq L(A)$ and $t=(t_1,\dots,t_r)\in \mathbb{R}^r$ with $||t||\leq A\sqrt{\log q}$, we have
\begin{equation*}
 \begin{aligned}
 &\hat{\mu}_{q;a_1,\dots,a_r}\left(\frac{t_1}{\sqrt{N_q}},\dots,\frac{t_r}{\sqrt{N_q}}\right)
 =\exp\left(-\frac{t_1^2+\cdots+t_r^2}{2}\right)
 \left(1+\frac{i}{\sqrt{N_q}}\sum_{j=1}^rC_q(a_j)t_j\right.\\
&\left. - \frac{1}{2N_q}\sum_{j=1}^rC_q(a_j)^2t_j^2
-\frac{1}{N_q}\left(\sum_{1\leq j<k\leq r}(B_q(a_j,a_k)+C_q(a_j)C_q(a_k))t_jt_k\right)+ \frac{Q_4(t_1,\dots, t_r)}{N_q}\right.\\
&\left.+ \sum_{m=0}^1\sum_{s=0}^2\sum_{\substack{0\leq l\leq L\\2l\geq 3-2s-m}}\frac{C_q^mB_q^l}{N_q^{m/2+l+s}}P_{s,m,l}(t_1,\dots,t_r) + O\left(\frac{r^{2L}B_q^L||t||^{2L}}{L!N_q^L}\right)\right),
 \end{aligned}
\end{equation*}
where $Q_4$ is a homogenous polynomial of degree $4$ with bounded coefficients and $P_{s,m,l}$ are homogenous polynomials of degree $m+2l+4s$ whose coefficients are bounded uniformly by a function of $l$. Moreover the constant in the $O$ is absolute.
\end{pro}
\begin{proof} For simplicity let us write $\hat{\mu}_q=\hat{\mu}_{q;a_1,\dots, a_r}$. From the explicit formula 2.1 we have
$$ \log\hat{\mu}_q\left(\frac{t_1}{\sqrt{N_q}},\dots,\frac{t_r}{\sqrt{N_q}}\right)= \frac{i}{\sqrt{N_q}}\sum_{j=1}^rC_q(a_j)t_j+\sum_{\substack{\chi\neq \chi_0\\ \chi\text{ mod } q}}\sum_{\gamma_{\chi}>0}\log J_0\left(\frac{2\left|\sum_{j=1}^r\chi(a_j)t_j\right|}
{\sqrt{\frac14+\gamma_{\chi}^2}\sqrt{N_q}}\right).$$
For $|s|\leq 1$ Lemma 2.8 of \cite{FiM} states that
$$ \log J_0(s)=-\sum_{n=1}^{\infty}u_{2n}s^{2n},$$
where $u_{2n}$ are positive real numbers with $u_2=1/4$ and $u_{2n}\ll (5/12)^{2n}.$ This implies that for $t=(t_1,\dots,t_r)$ with $||t||\leq A\sqrt{\log q}$ we have
\begin{equation}
\log\hat{\mu}_q\left(\frac{t_1}{\sqrt{N_q}},\dots,\frac{t_r}{\sqrt{N_q}}\right)= \frac{i}{\sqrt{N_q}}\sum_{j=1}^rC_q(a_j)t_j-\sum_{n=1}^{\infty} \frac{u_{2n}2^{2n}}{N_q^n}\sum_{\substack{\chi\neq \chi_0\\ \chi\text{ mod } q}}\sum_{\gamma_{\chi}>0}\frac{\left|\sum_{j=1}^r\chi(a_j)t_j\right|^{2n}}{(\frac14+\gamma_{\chi}^2)^n}.
\end{equation}
The contribution of the term $n=1$ to the RHS of (3.8) equals
\begin{equation}
-\frac{1}{N_q}\sum_{\substack{\chi\neq \chi_0\\ \chi\text{ mod } q}}\sum_{\gamma_{\chi}>0}\frac{1}{\frac14+\gamma_{\chi}^2}\sum_{1\leq j,k\leq r}\chi(a_j)\overline{\chi(a_k)}t_jt_k=-\frac{1}{2}(t_1^2+\dots+t_r^2)
-\frac{1}{N_q}\sum_{1\leq j<k\leq r}B_q(a_j,a_k)t_jt_k.
\end{equation}
The term $n=2$ contributes $Q_4(t_1,\dots,t_r)/N_q$ where
\begin{equation*}
 \begin{aligned}
 Q_4(t_1,\dots,t_r):&=-\frac{16u_4}{N_q}\sum_{\substack{\chi\neq \chi_0\\ \chi\text{ mod } q}}\sum_{\gamma_{\chi}>0}\frac{\left|\sum_{j=1}^r\chi(a_j)t_j\right|^{4}}{(\frac14+\gamma_{\chi}^2)^2}\\
 &= -\frac{16u_4}{N_q}\sum_{1\leq j_1,j_2,j_3,j_4\leq r}\sum_{\substack{\chi\neq \chi_0\\ \chi\text{ mod } q}}\sum_{\gamma_{\chi}>0}\frac{\chi(a_{j_1})\chi(a_{j_2})\overline{\chi}(a_{j_3})\overline{\chi}(a_{j_4})}{(\frac14+\gamma_{\chi}^2)^2}
 t_{j_1}t_{j_2}t_{j_3}t_{j_4}.\\
\end{aligned}
\end{equation*}
Since $\sum_{\substack{\chi\neq \chi_0\\ \chi\text{ mod } q}}\sum_{\gamma_{\chi}>0}1/(\frac{1}{4}+\gamma_{\chi}^2)^2\leq 2 N_q$, then $Q_4(t_1,\dots,t_r)$ is a homogenous polynomial of degree $4$ with bounded coefficients.  Furthermore, the contribution of the terms $n\geq 3$ to the RHS of (3.8) is $\ll ||t||^6/N_q^2\ll_r (\log q)/\phi(q)^2$, which follows from our assumption on $t$ along with the fact that $\sum_{\substack{\chi\neq \chi_0\\ \chi\text{ mod } q}}\sum_{\gamma_{\chi}>0}1/(\frac{1}{4}+\gamma_{\chi}^2)^n\leq 4^n N_q.$ Hence, we obtain
\begin{equation}
\begin{aligned}
\log\hat{\mu}_q\left(\frac{t_1}{\sqrt{N_q}}
,\dots,\frac{t_r}{\sqrt{N_q}}\right)&= -\frac{1}{2}(t_1^2+\dots+t_r^2)+ \frac{i}{\sqrt{N_q}}\sum_{j=1}^rC_q(a_j)t_j\\
&-\frac{1}{N_q}\sum_{1\leq j<k\leq r}B_q(a_j,a_k)t_jt_k +\frac{Q_4(t_1,\dots,t_r)}{N_q}+O_r\left(\frac{\log q}{\phi(q)^2}\right).\\
\end{aligned}
\end{equation}
Now, in our range of $t$ we have
$$ \exp\left(\frac{i}{\sqrt{N_q}}\sum_{j=1}^rC_q(a_j)t_j\right)
=\sum_{m=0}^2\frac{1}{m!N_q^{m/2}}\left(i\sum_{j=1}^rC_q(a_j)t_j\right)^m
+O_r\left(\frac{C_q^3}{\phi(q)^{3/2}}\right),$$
and  $\exp(Q_4(t_1,\dots,t_r)/N_q)=1+Q_4(t_1,\dots,t_r)/N_q+O_r(\log^2 q/\phi(q)^2).$ Therefore, using that
$$ \exp\left(-\frac{1}{N_q}\sum_{1\leq j<k\leq r}B_q(a_j,a_k)t_jt_k\right)=\sum_{l=0}^{\infty}
\frac{\left(-\sum_{1\leq j<k\leq r}B_q(a_j,a_k)t_jt_k\right)^l}{l!N_q^l},$$
along with the previous estimates and equation (3.10) we deduce that the quotient of $\hat{\mu}_q\left(\frac{t_1}{\sqrt{N_q}},
\dots,\frac{t_r}{\sqrt{N_q}}\right)$ and $\exp(-(t_1^2+\cdots+t_r^2)/2)$ equals
\begin{equation}
\sum_{s=0}^1\sum_{m=0}^2\sum_{l=0}^{\infty}
\frac{Q_4(t_1,\dots,t_r)^s}{m!l!N_q^{m/2+l+s}}
\left(i\sum_{j=1}^rC_q(a_j)t_j\right)^m\left(-\sum_{1\leq j<k\leq r}B_q(a_j,a_k)t_jt_k\right)^l+O_r\left(\frac{C_q^3}{\phi(q)^{3/2}}\right).
\end{equation}
We collect the summands above according to $D=m+2s+2l$ (which equals twice the power of $N_q$). Then, it is easy to check that the contribution of the terms $0\leq D\leq 2$ to the main term of (3.11) equals
$$ 1+\frac{i}{\sqrt{N_q}}\sum_{j=1}^rC_q(a_j)t_j
-\frac{1}{2N_q}\left(\sum_{j=1}^rC_q(a_j)t_j\right)^2
-\frac{1}{N_q}\sum_{1\leq j<k\leq r}B_q(a_j,a_k)t_jt_k+\frac{Q_4(t_1,\dots,t_r)}{N_q}.$$
Let $P_{s,m,l}(t_1,\dots,t_r)$ be the homogenous polynomial of degree $m+2l+4s$ defined by
$$ P_{s,m,l}(t_1,\dots,t_r)=\frac{1}{m!l!}C_q^{-m}B_q^{-l} Q_4(t_1,\dots,t_r)^s\left(i\sum_{j=1}^r
C_q(a_j)t_j\right)^m\left(-\sum_{1\leq j<k\leq r}B_q(a_j,a_k)t_jt_k\right)^l.$$
Then the contribution of the terms with $D\geq 3$ to (3.11) equals
$$ \sum_{m=0}^1\sum_{s=0}^2\sum_{\substack{l\geq 0\\2l\geq 3-2s-m}}\frac{C_q^mB_q^l}{N_q^{m/2+l+s}}P_{s,m,l}
(t_1,\dots,t_r).$$
Notice that the coefficients of $P_{s,m,l}$ are bounded uniformly by a function of $l$ since $r$ is fixed and $s,m\leq 2$.
On the other hand since $C_q=q^{o(1)}$ we get
$$ \frac{C_q^mB_q^l}{N_q^{m/2+l+s}}P_{s,m,l}(t_1,\dots,t_r)\ll \frac{r^{2l+m}C_q^mB_q^l||t||^{m+2l+4s}}{m!l!N_q^{m/2+l+s}}\ll \frac{r^{2l}B_q^l||t||^{2l}}{l!N_q^l}.$$
Now, Corollary 5.4 implies that $B_q\leq c\phi(q)$ for some absolute constant $c>0$. Therefore, in our range of $t$, we have $r^{2}B_q||t||^{2}/(lN_q)\leq 2r^2A/l.$
This shows that for a suitably large constant $L(A)$ (which also depends on $r$)  we have
$$ \sum_{m=0}^1\sum_{s=0}^2\sum_{\substack{l>L\\2l\geq 3-2s-m}}\frac{C_q^mB_q^l}{N_q^{m/2+l+s}}P_{s,m,l}(t_1,\dots,t_r)\ll \frac{r^{2L}B_q^L||t||^{2L}}{L!N_q^L},$$
for all $L\geq L(A)$, completing  the proof.
\end{proof}

\section{An asymptotic formula for the densities $\delta_{q; a_1,\dots,a_r}$}

The first step to prove Theorem 1 is to truncate the integral on the RHS of (1.1). To this end we need to bound the tail of the distribution $\mu_{q; a_1,\dots,a_r}$. Our idea consists in relating this tail to the Laplace transform of $\mu_{q;a_1,\dots, a_r}$ using Chernoff's bound. For $s=(s_1,s_2,...,s_r)\in \mathbb{R}^r$ we define
$$
 \mathcal{L}_{q;a_1,\dots, a_r}(s_1,s_2,\dots,s_r):=\int_{x\in \mathbb{R}^r}e^{s_1x_1+\cdots+s_rx_r}d\mu_{q;a_1,\dots, a_r}(x_1,\dots,x_r),
$$
if this integral converges.
The same arguments as in the proof of
Rubinstein and Sarnak for the explicit formula (2.1) of $\hat{\mu}_{q;a_1,\dots, a_r}$, show under GRH and GSH, that $\mathcal{L}_{q;a_1,\dots, a_r}(s)$ exists for all $s\in\mathbb{R}^r$ and
\begin{equation}
\mathcal{L}_{q;a_1,\dots, a_r}(s_1,s_2,\dots,s_r)= \exp\left(-\sum_{j=1}^rC_q(a_j)s_j\right)\prod_{\substack{\chi\neq \chi_0\\\chi\text{ mod } q}}\prod_{\gamma_{\chi}>0}I_0
\left(\frac{2|\sum_{j=1}^r\chi(a_i)s_i|}{\sqrt{\frac14+\gamma_{\chi}^2}}\right),
\end{equation}
where $I_0(t):= \sum_{n=0}^{
\infty}(t/2)^{2n}/n!^2$ is the modified Bessel function of order $0$. We prove

\begin{pro} Assume GRH and GSH. Fix an integer $r\geq 2$ and let $q$ be a large positive integer. Then for $R\geq \sqrt{\phi(q)\log q}$ we have
$$\mu_{q;a_1,\dots, a_r}(|x|_{\infty}>R)\leq \exp\left(-\frac{R^2}{2\phi(q)\log q}\left(1+O\left(\frac{\log\log q}{\log q}\right)\right)\right).$$
\end{pro}
\begin{proof} First we note that
$$ \mu_{q;a_1,\dots, a_r}(|x|_{\infty}>R) \leq \sum_{j=1}^r \mu_{q;a_1,\dots, a_r}(x_j>R)+ \sum_{j=1}^r\mu_{q;a_1,\dots, a_r}(x_j<-R).
$$  We shall only bound $\mu_{q;a_1,\dots, a_r}(x_j>R)$, since the corresponding bound for $\mu_{q;a_1,\dots, a_r}(x_j<-R)$ can be obtained similarly. Let $s>0$. Then using (4.1) we get
\begin{equation*}
\begin{aligned}
\mu_{q;a_1,\dots, a_r}(x_j>R)&\leq e^{-sR} \int_{(x_1,\dots,x_r)\in \mathbb{R}^r} e^{sx_j}d\mu_{q;a_1,\dots,a_r}(x_1,\dots,x_r)\\
&\leq e^{-sR-sC_q(a_j)}\prod_{\substack{\chi\neq \chi_0\\\chi\text{ mod } q}}\prod_{\gamma_{\chi}>0}I_0
\left(\frac{2s}{\sqrt{\frac14+\gamma_{\chi}^2}}\right).
\end{aligned}
\end{equation*}
Since $I_0(s)\leq \exp(s^2/4)$ for all $s\in \mathbb{R}$ we obtain
\begin{equation*}
\begin{aligned}
 \prod_{\substack{\chi\neq \chi_0\\\chi\text{ mod } q}}\prod_{\gamma_{\chi}>0}I_0
\left(\frac{2s}{\sqrt{\frac14+\gamma_{\chi}^2}}\right)&\leq \exp\left(s^2\sum_{\substack{\chi\neq \chi_0\\\chi\text{ mod } q}}\sum_{\gamma_{\chi}>0}\frac{1}{{\frac14+\gamma_{\chi}^2}}\right)\\
&\leq \exp\left(\frac{s^2\phi(q)\log q}{2}\left(1+O\left(\frac{\log\log q}{\log q}\right)\right)\right),
\end{aligned}
\end{equation*}
by Lemma 3.1.
The result follows by taking $s=R/(\phi(q)\log q)$ along with the fact that $C_q(a_j)\ll_{\epsilon} q^{\epsilon}$ for any $\epsilon>0$.
\end{proof}

Let $\Phi(x):=e^{-x^2/2}$ and denote by $\Phi^{(n)}$ the $n$-th derivative of $\Phi$. Then $\Phi^{(1)}(x)=-xe^{-x^2/2}$, $\Phi^{(2)}(x)=(x^2-1)e^{-x^2/2}$, and more generally we know that
$\Phi^{(n)}(x)= (-1)^nH_n(x)e^{-x^2/2}$ where $H_n$ is the $n$-th Hermite polynomial. The last ingredients we need in order to prove Theorem 1 are the following lemmas
\begin{lem}
Let $n_1,\dots, n_r$ be fixed non-negative integers, and $M$ be a large positive number. Then for any  $(x_1,\dots,x_r)\in \mathbb{R}^r,$ we have
$$ \int_{||\mathbf{t}||<M}e^{i(t_1x_1+\cdots+t_rx_r)}
\prod_{j=1}^rt_j^{n_j}\Phi(t_j)
d\mathbf{t}= (2\pi)^{r/2}\prod_{j=1}^r i^{n_j}H_{n_j}(x_j)e^{-x_j^2/2} +O\left(e^{-M^2/4}\right).$$
\end{lem}
\begin{proof} First, notice that
\begin{equation*}
\int_{\mathbf{t}\in \mathbb{R}^r}e^{i(t_1x_1+\cdots+t_rx_r)}
\prod_{j=1}^r t_j^{n_j}\Phi(t_j)d\mathbf{t} =(2\pi)^{r/2}\prod_{j=1}^r\Psi_j(x_j),
\end{equation*}
where $$\Psi_j(u)= \frac{1}{\sqrt{2\pi}}\int_{-\infty}^{\infty}e^{iuv}\Phi(v)v^{n_j}dv.$$
Since the Fourier transform of $\Phi(u)/(2\pi)$ is $\Phi(v)/\sqrt{2\pi}$, then using standard properties of the Fourier transform, we deduce that $\Phi(v)v^{n_j}/\sqrt{2\pi}$ is the Fourier transform of $\frac{(-i)^{n_j}}{2\pi}\Phi^{(n_j)}(v)$. Therefore the Fourier inversion formula gives
$$ \Psi_j(u)=(-i)^{n_j}\Phi^{(n_j)}(u)= i^{n_j}H_{n_j}(u)e^{-u^2/2}.$$
Finally, note that
$$ \int_{||\mathbf{t}||> M}\left|\prod_{j=1}^rt_j^{n_j}\Phi(t_j)\right|d\mathbf{t}\ll \exp\left(-M^2/2\right)M^{n_1+\cdots+n_r}\ll \exp\left(-M^{2}/4\right),$$
if $M$ is large enough, which completes the proof.
\end{proof}

\begin{lem}  Let $P_n(t_1,\dots,t_r)$ be a homogeneous polynomial of degree $n$, whose coefficients are complex numbers uniformly bounded by a function of $n$. Let $R$ be a large positive number and $M\geq \log R$ be a real number. Then we have
$$\left|\int_{\substack{x_1>x_2>\dots>x_r\\ |\mathbf{x}|_{\infty}<R}}\int_{||\mathbf{t}||\leq
M}e^{i(t_1x_1+\cdots+t_rx_r)}
\exp\left(-\frac{t_1^2+\cdots+t_r^2}{2}\right)
P_n(t_1,\dots,t_r)d\mathbf{t}d\mathbf{x}\right|\ll_{n,r}1.$$
\end{lem}
\begin{proof}
Since the coefficients of $P_n(t_1,\dots,t_r)$ are uniformly bounded by a function of $n$, it is sufficient to show that the statement holds when $P_n(t_1,\dots, t_r)=t_1^{n_1}\dots t_r^{n_r},$ where $n_i$ are non-negative integers with $n_1+\cdots+n_r=n$. Using Lemma 4.2 we get
\begin{equation}
\begin{aligned}
&\int_{\substack{x_1>x_2>\dots>x_r\\ |\mathbf{x}|_{\infty}<R}}\int_{||\mathbf{t}||\leq M}e^{i(t_1x_1+\cdots+t_rx_r)}
\exp\left(-\frac{t_1^2+\cdots+t_r^2}{2}\right)t_1^{n_1}\dots t_r^{n_r}d\mathbf{t}d\mathbf{x}\\
&=i^{n}(2\pi)^{r/2}\int_{\substack{x_1>x_2>\dots>x_r\\ |\mathbf{x}|_{\infty}<R}}H_{n_1}(x_1)\dots H_{n_r}(x_r)\exp\left(-\frac{x_1^2+\cdots+x_r^2}{2}\right)d\mathbf{x}+o_R(1),\\
\end{aligned}
\end{equation}
since $\exp\left(-M^2/4\right)R^r\ll e^{-\frac{\log^2R}{8}}$ by our hypothesis on $M$. The lemma then follows upon noting that
$$ \int_{\substack{x_1>x_2>\dots>x_r\\ |\mathbf{x}|_{\infty}>R}}H_{n_1}(x_1)\dots H_{n_r}(x_r)\exp\left(-\frac{x_1^2+\cdots+x_r^2}{2}\right)dx_1\dots dx_r\ll_{n,r} R^{n}e^{-R^2/2}=o_R(1),$$
and
$$ \int_{x_1>x_2>\dots>x_r}H_{n_1}(x_1)\dots H_{n_r}(x_r)\exp\left(-\frac{x_1^2+\cdots+x_r^2}{2}\right)dx_1\dots dx_r\ll _{n,r}1.$$
\end{proof}
\begin{proof}[Proof of Theorem 1]
Let $R:=\sqrt{N_q}\log q.$ To lighten the notation in this proof we write $\delta_q$ for $\delta_{q;a_1,\dots, a_r}$ and $\mu_q$ for $\mu_{q;a_1,\dots,a_r}$. Then by Proposition 4.1 we obtain
\begin{equation}
\delta_q= \int_{y_1>y_2>\dots>y_r}d\mu_q(y_1, \dots, y_r)= \int_{\substack{y_1>y_2>\dots>y_r\\ |\mathbf{y}|_{\infty}\leq R}}d\mu_q(y_1, \dots, y_r)+ O\left(\exp\left(-\frac{\log^2 q}{10}\right)\right).
\end{equation}
Next, we apply the Fourier inversion formula to the measure $\mu_q$ to get
$$ \int_{\substack{y_1>y_2>\dots>y_r\\ |\mathbf{y}|_{\infty}\leq R}}d\mu_q(y_1, \dots, y_r)
=(2\pi)^{-r}\int_{\substack{y_1>y_2>\dots>y_r\\ |\mathbf{y}|_{\infty}\leq R}}\int_{\mathbf{s}\in \mathbb{R}^r}e^{i(s_1y_1+\cdots+s_ry_r)}\hat{\mu}_q(s_1,\dots, s_r)d\mathbf{s}d\mathbf{y}.$$
Let $A= A(r)\geq r$ be a suitably large constant. Then using Proposition 3.2 with $\epsilon:= A(N_q)^{-1/2}\sqrt{\log q}$ implies
$$  \int_{\mathbf{s}\in \mathbb{R}^r}e^{i(s_1y_1+\cdots+s_ry_r)}\hat{\mu}_q(s_1,\dots, s_r)d\mathbf{s}= \int_{||\mathbf{s}||\leq \epsilon}e^{i(s_1y_1+\cdots+s_ry_r)}\hat{\mu}_q(s_1,\dots, s_r)d\mathbf{s}+ O\left(\frac{1}{q^{2A}}\right).$$
Hence we obtain
\begin{equation}\delta_q= (2\pi)^{-r}\int_{\substack{y_1>y_2>\dots>y_r\\ |\mathbf{y}|_{\infty}\leq R}}\int_{||\mathbf{s}||\leq \epsilon}e^{i(s_1y_1+\cdots+s_ry_r)}\hat{\mu}_q(s_1,\dots, s_r)d\mathbf{s}d\mathbf{y}+ O\left(\frac{1}{q^{A}}\right),
\end{equation}
using that $R^rq^{-2A}\ll q^{-A}.$ Upon making the change of variables
$$ t_j:=  \sqrt{N_q} s_j, \text{ and } x_j:= \frac{y_j}{\sqrt{N_q}},$$
we infer from (4.4) that
\begin{equation} \delta_q= (2\pi)^{-r}\int_{\substack{x_1>x_2>\dots>x_r\\ |\mathbf{x}|_{\infty}\leq \log q}}\int_{||\mathbf{t}|| \leq A\sqrt{\log q} }e^{i(t_1x_1+\cdots+t_rx_r)}\hat{\mu}_q\left(\frac{t_1}{\sqrt{N_q}},\dots, \frac{t_r}{\sqrt{N_q}}\right)d\mathbf{t}d\mathbf{x}+ O\left(\frac{1}{q^{A}}\right).
\end{equation}
Now we use the asymptotic expansion of $\hat{\mu}_q\left(t_1N_q^{-1/2},\dots, t_rN_q^{-1/2}\right)$ proved in Proposition 3.3. We take $L=L(A)\geq 2r$ to be a suitably large constant. Then, Lemma 4.3 shows that the contribution of the error term along with the terms corresponding to the polynomials $Q_4$ and $P_{s,m,l}$ (in the asymptotic expansion of Proposition 3.3) to the integral on the RHS of (4.5) is
\begin{equation}
\ll_r \frac{1}{N_q}+ \sum_{m=0}^1\sum_{s=0}^2\sum_{\substack{0\leq l\leq L\\2l\geq 3-2s-m}}\frac{C_q^mB_q^l}{N_q^{m/2+l+s}}+\frac{(\log q)^rB_q^L}{N_q^L}\ll_r
\frac{1}{N_q}+ \frac{C_qB_q}{N_q^{3/2}}+ \frac{B_q^2}{N_q^2},
\end{equation}
since $B_q\ll N_q/\log q$ by Corollary 5.4.
Now we shall compute the contribution of the remaining terms in the asymptotic formula of $\hat{\mu}_q$ to the integral in (4.5). Appealing to Lemma 4.2 along with the fact that  $\exp(-(x_1^2+\cdots+x_r^2)/2)$ is a continuous symmetric function in $x_1,\dots,x_r$, we obtain
\begin{equation}
\begin{aligned}
& (2\pi)^{-r}\int_{\substack{x_1>x_2>\dots>x_r\\ |\mathbf{x}|_{\infty}\leq \log q}}\int_{||\mathbf{t}|| \leq A\sqrt{\log q} }e^{i(t_1x_1+\cdots+t_rx_r)}\exp\left(-\frac{t_1^2+\cdots+t_r^2}{2}\right)d\mathbf{t}d\mathbf{x}\\
&=(2\pi)^{-r/2}\int_{\substack{x_1>x_2>\dots>x_r\\ |\mathbf{x}|_{\infty}\leq \log q}}\exp\left(-\frac{x_1^2+\cdots+x_r^2}{2}\right) d\mathbf{x}+O\left(\frac{1}{q^A}\right)\\
&= \frac{1}{r!(2\pi)^{r/2}}\int_{\mathbf{x}\in \mathbb{R}^r}\exp\left(-\frac{x_1^2+\cdots+x_r^2}{2}\right)d\mathbf{x}+O\left(\frac{1}{q^A}\right)= \frac{1}{r!}+O\left(\frac{1}{q^A}\right).\\
\end{aligned}
\end{equation}
Similarly, we infer from Lemma 4.2 that for $1\leq j\leq r$, we have
\begin{equation}
\begin{aligned}
&(2\pi)^{-r}\int_{\substack{x_1>x_2>\dots>x_r\\ |\mathbf{x}|_{\infty}\leq \log q}}\int_{||\mathbf{t}|| \leq A\sqrt{\log q} }t_je^{i(t_1x_1+\cdots+t_rx_r)}\exp\left(-\frac{t_1^2+\cdots+t_r^2}{2}\right)d\mathbf{t}d\mathbf{x}\\
&= i(2\pi)^{-r/2}\int_{\substack{x_1>x_2>\dots>x_r\\ |\mathbf{x}|_{\infty}\leq \log q}}x_j\exp\left(-\frac{x_1^2+\cdots+x_r^2}{2}\right) d\mathbf{x}+O\left(\frac{1}{q^A}\right)= i\alpha_j(r)+  O\left(\frac{1}{q^A}\right),\\
\end{aligned}
\end{equation}
and
\begin{equation}
\begin{aligned}
&(2\pi)^{-r}\int_{\substack{x_1>x_2>\dots>x_r\\ |\mathbf{x}|_{\infty}\leq \log q}}\int_{||\mathbf{t}|| \leq A\sqrt{\log q} }t_j^2e^{i(t_1x_1+\cdots+t_rx_r)}\exp\left(-\frac{t_1^2+\cdots+t_r^2}{2}\right)d\mathbf{t}d\mathbf{x}\\
&= -\lambda_j(r)+O\left(\frac{1}{q^A}\right).\\
\end{aligned}
\end{equation}
For $1\leq j<k\leq r$ we analogously obtain
\begin{equation}
\begin{aligned}
 &(2\pi)^{-r}\int_{\substack{x_1>x_2>\dots>x_r\\ |\mathbf{x}|_{\infty}\leq \log q}}\int_{||\mathbf{t}|| \leq A\sqrt{\log q} }t_jt_ke^{i(t_1x_1+\cdots+t_rx_r)}\exp\left(-\frac{t_1^2+\cdots+t_r^2}{2}\right)d\mathbf{t}d\mathbf{x}\\
 &= -\beta_{j,k}(r)+O\left(\frac{1}{q^A}\right).\\
 \end{aligned}
\end{equation}
The theorem now follows upon combining Proposition 3.3 with the estimates (4.5)-(4.10).
\end{proof}
 In the remaining part of this section, we explicitly compute the constants $\alpha_j(r)$ and $\beta_{j,k}(r)$ for $r=3$. To simplify the computations we prove the following identities
\begin{lem} Let $r\geq 2$. Then for any $1\leq j<k\leq r$ we have $$ \alpha_j(r)= -\alpha_{r+1-j}(r) \text{ and } \beta_{j,k}(r)= \beta_{r+1-k,r+1-j}(r).$$
\end{lem}
\begin{proof} We only prove the identity for the $\alpha_j(r)$ since the proof for the $\beta_{j,k}(r)$ is similar. Recall that
$$ \alpha_j(r) =(2\pi)^{-r/2}\int_{x_1>x_2>\dots>x_r}x_j\exp\left(-\frac{x_1^2+\cdots+x_r^2}{2}\right) dx_1\dots dx_r.$$
Upon making the change of variables $y_j= -x_{r+1-j}$ we deduce that
$$ \alpha_j(r)=-(2\pi)^{-r/2}\int_{y_1>y_2>\dots>y_r}y_{r+1-j}\exp\left(-\frac{y_1^2+\cdots+y_r^2}{2}\right) dy_1\dots dy_r=-\alpha_{r+1-j}(r).$$
\end{proof}
\begin{lem} We have $\beta_{1,2}(2)=0.$
Moreover, one has
$$ \alpha_1(3)= \frac{1}{4\sqrt{\pi}}, \alpha_2(3)= 0,  \alpha_3(3)= -\frac{1}{4\sqrt{\pi}},$$
and
$$ \beta_{1,2}(3)= \beta_{2,3}(3)= \frac{1}{4\pi\sqrt{3}}, \beta_{1,3}(3)= -\frac{1}{2\pi\sqrt{3}}.$$
\end{lem}
\begin{proof}
First we have
$$
\beta_{1,2}(2)=\frac{1}{2\pi} \int_{-\infty}^{\infty} \int_{x_2}^{\infty} x_1x_2\exp\left(-\frac{x_1^2+x_2^2}{2}\right)dx_1dx_2
= \frac{1}{2\pi}\int_{-\infty}^{\infty} x_2e^{-x_2^2}dx_2=0. $$
Now we deal with the case $r=3$. Recall that
\begin{equation}
\sum_{j=1}^3\alpha_j(3)= \sum_{1\leq j<k\leq 3}\beta_{j,k}(3)=0.
\end{equation}
We begin by computing
$$
\alpha_1(3)= \frac{1}{(2\pi)^{3/2}}\int_{x_1>x_2>x_3} x_1 \exp\left(-\frac{x_1^2+x_2^2+x_3^2}{2}\right)dx_1dx_2dx_3.
$$
To this end, we integrate with respect to $x_1$ first to get
$$ \alpha_1(3)= \frac{1}{(2\pi)^{3/2}}\int_{x>y} \exp\left(-x^2-\frac{y^2}{2}\right)dxdy= \frac{1}{(2\pi)^{3/2}}\int_{X<Y} \exp\left(-X^2-\frac{Y^2}{2}\right)dXdY,
$$
by making the change of variables $X=-x$ and $Y=-y$. Hence, we deduce that
$$ \alpha_1(3)= \frac{1}{2(2\pi)^{3/2}}\int_{-\infty}^{\infty}\int_{-\infty}^{\infty}
\exp\left(-x^2-\frac{y^2}{2}\right)dxdy=\frac{1}{4\sqrt{\pi}}.
$$
On the other hand, Lemma 4.4 shows that $\alpha_3(3)=-\alpha_1(3)$, and this combined with equation (4.11) leads to $\alpha_2(3)=0$. Furthermore, we have
$$
\beta_{1,2}(3)=\frac{1}{(2\pi)^{3/2}}\int_{x_1>x_2>x_3} x_1x_2 \exp\left(-\frac{x_1^2+x_2^2+x_3^2}{2}\right)dx_1dx_2dx_3.
$$
Performing the integration with respect to $x_1$ first, then with respect to $x_2$ gives us
$$ \beta_{1,2}(3)= \frac{1}{2(2\pi)^{3/2}}
\int_{-\infty}^{\infty}e^{-\frac{3x_3^2}{2}}dx_3= \frac{1}{4\pi\sqrt{3}}.$$
The remaining estimates follow upon using Lemma 4.4 to get $\beta_{2,3}(3)=\beta_{1,2}(3)$, and then applying equation (4.11) to deduce that $\beta_{1,3}(3)=-2\beta_{1,2}(3).$
\end{proof}

\section{The average order of $|B_q(a,b)|$}

In this section we prove upper and lower bounds (of the same order of magnitude) for the first moment of $|B_q(a,b)|$ over pairs of residue classes $(a,b)\in \mathcal{A}_2(q)$.
To this end, we begin by proving the following key proposition
\begin{pro} Assume GRH. Let $q$ be a large integer, and $(a,b)\in \mathcal{A}_2(q)$. Put $x=(q\log q)^2$. Then we have
\begin{equation*}
\begin{aligned}
B_q(a,b)&=4\log q-\phi(q)l_q(a,b)\log 2-\phi(q)\frac{\Lambda\left(\frac{q}{(q,a-b)}\right)}{\phi\left(\frac{q}{(q,a-b)}\right)}-\phi(q)\sum_{\substack{n\leq 2x\log x\\ bn\equiv a \text{ mod }q}}\frac{\Lambda(n)}{n}e^{-n/x}\\
&-\phi(q)\sum_{\substack{n\leq 2x\log x\\ an\equiv b \text{ mod }q}}\frac{\Lambda(n)}{n}e^{-n/x}-\phi(q)\sum_{p^{\nu}\parallel q}\sum_{\substack{1\leq e\leq 2\log x\\ap^e\equiv b  \text{ mod } q/p^{\nu}}}\frac{\log p}{p^{e+\nu-1}(p-1)}\\
& -\phi(q)\sum_{p^{\nu}\parallel q}\sum_{\substack{1\leq e\leq 2\log x\\bp^e\equiv a  \text{ mod } q/p^{\nu}}}\frac{\log p}{p^{e+\nu-1}(p-1)}+O(\log\log q),
\end{aligned}
\end{equation*}
where $l_q(a,b)=1$ if $a+b\equiv 0\text{ mod } q$ and $0$ otherwise.
\end{pro}

\emph{Remark 5.1} This result implies that $B_q(a,b)<0$ if $|B_q(a,b)|>5\log q$.

Although the major part of this proposition is proved in \cite{FiM} (see Theorems 1.4 and 1.7 there), we chose to include the details of the proof for the seek of completeness. The only new input is the following lemma which corresponds to the contribution of the principal character $\chi_0$ mod $q$.
\begin{lem} Let $q$ be a large positive integer and $y\geq q$ be a real number. Then
$$ \sum_{\substack{n\geq 1\\(n,q)=1}}\frac{\Lambda(n)}{n}e^{-n/y}=\log y+ O(\log\log y).$$
\end{lem}
\begin{proof} First note that
$$ \sum_{\substack{n\geq 1\\(n,q)>1}}\frac{\Lambda(n)}{n}e^{-n/y}\leq  \sum_{p|q}\sum_{k=1}^{\infty} \frac{\log p}{p^k}= \sum_{p|q}\frac{\log p}{p-1}\ll \log\log q.$$
Thus it suffices to evaluate
$$ \sum_{n=1}^{\infty}\frac{\Lambda(n)}{n}e^{-n/y}.$$
We split the above sum into three parts: $n>y\log^2y$, $y\log\log y< n\leq y\log^2y$ and finally $n\leq y\log\log y.$ The contribution of the first part is
$$ \sum_{n> y\log^2y}\frac{\Lambda(n)}{n}e^{-n/y}\leq \sum_{n> y\log^2y}\frac{1}{n^2}\leq \frac{1}{y},$$
which follows from the fact that $e^{-n/y}\leq n^{-2}$ for $n>y\log^2y$. Now the contribution of the second part is
$$ \sum_{y\log\log y< n\leq y\log^2y}\frac{\Lambda(n)}{n}e^{-n/y}\leq \frac{1}{\log y}\sum_{n\leq y\log^2y}\frac{\Lambda(n)}{n}\ll 1. $$
Finally using that $1-e^{-t}\leq 2t$ for all $t>0$, we deduce that the contribution of the last part equals
$$ \sum_{n\leq y\log\log y} \frac{\Lambda(n)}{n}e^{-n/y}= \sum_{n\leq y\log\log y} \frac{\Lambda(n)}{n}+ O\left(\frac{1}{y}\sum_{n\leq y\log\log y} \Lambda(n)\right)=\log y+ O(\log\log y),$$
which follows from the prime number theorem.
\end{proof}
\begin{proof}[Proof of Proposition 5.1]
 Let $(a,b)\in \mathcal{A}_2(q)$. First we infer from (3.1) that
\begin{equation}
 \begin{aligned}
 B_q(a,b)&=\sum_{\substack{\chi\neq \chi_0\\ \chi\text{ mod } q}}\sum_{\gamma_{\chi}>0}\frac{\chi\left(\frac{a}{b}\right)+ \chi\left(\frac{b}{a}\right)}{\frac{1}{4}+\gamma_{\chi}^2}= \frac{1}{2} \sum_{\substack{\chi\neq \chi_0\\ \chi\text{ mod } q}}\sum_{\gamma_{\chi}}\frac{\chi\left(\frac{a}{b}\right)+ \chi\left(\frac{b}{a}\right)}{\frac{1}{4}+\gamma_{\chi}^2}\\
 &= \frac12\sum_{\substack{\chi\neq \chi_0\\ \chi\text{ mod } q}}\left(\chi\left(\frac{a}{b}\right)+ \chi\left(\frac{b}{a}\right)\right)\log q^{*}_{\chi}-\phi(q)l_q(a,b)\log2 \\
 & + \sum_{\substack{\chi\neq \chi_0\\ \chi\text{ mod } q}}\left(\chi\left(\frac{a}{b}\right)+ \chi\left(\frac{b}{a}\right)\right)\text{Re}
 \frac{L'(1,\chi^{*})}{L(1,\chi^{*})}+O(1),
 \end{aligned}
\end{equation}
using the orthogonality relations for characters. In order to evaluate the first sum on the RHS of (5.1) we use equation (3.2) which gives
\begin{equation}
\frac12\sum_{\substack{\chi\neq \chi_0\\ \chi\text{ mod } q}}(\chi(a/b)+ \chi(b/a))\log q^{*}_{\chi}= -\phi(q)\frac{\Lambda\left(\frac{q}{(q,a-b)}\right)}{\phi\left(\frac{q}{(q,a-b)}\right)}.
\end{equation}
Now we compute the sum over the $L$-values. First we record a standard approximation formula for $L'/L(1,\chi^{*})$ under GRH, which corresponds to Proposition 3.10 of \cite{FiM}
\begin{equation}
 \frac{L'(1,\chi^{*})}{L(1,\chi^{*})}= -\sum_{n=1}^{\infty}\frac{\chi^{*}(n)\Lambda(n)}{n}e^{-n/y} + O\left(\frac{\log q}{y^{1/2}}\right).
\end{equation}
Inserting this estimate into the second sum on the RHS of (5.1), we obtain
\begin{equation}
\begin{aligned}
&\sum_{\substack{\chi\neq \chi_0\\ \chi\text{ mod } q}}\left(\chi\left(\frac{a}{b}\right)+ \chi\left(\frac{b}{a}\right)\right)\text{Re}\frac{L'(1,\chi^{*})}{L(1,\chi^{*})}= \text{Re}\sum_{\substack{\chi\neq \chi_0\\ \chi\text{ mod } q}}\left(\chi\left(\frac{a}{b}\right)+ \chi\left(\frac{b}{a}\right)\right)\frac{L'(1,\chi^{*})}{L(1,\chi^{*})}\\
&= -\text{Re}\sum_{n=1}^{\infty}\frac{\Lambda(n)}{n}e^{-n/y}\sum_{\substack{\chi\neq \chi_0\\ \chi\text{ mod } q}}(\chi(a/b)\chi^{*}(n)+\chi(b/a)\chi^{*}(n))
+ O\left(\frac{\phi(q)\log q}{y^{1/2}}\right).\\
\end{aligned}
\end{equation}
Let $p$ be a prime number and $e\geq 1$ a positive integer. To evaluate the inner sum over characters in the RHS of (5.4) we use Proposition 3.4 of \cite{FiM} which states that
\begin{equation}
\sum_{\chi\text{ mod }q} \chi\left(a/b\right) \chi^{*}(p^e)= \left\{
\begin{aligned}&\phi(q)  \ \ \ \ \ \ \text{ if } p\nmid q \text{ and } ap^e\equiv b \text{ mod } q,\\
& \phi(q/p^{\nu}) \ \  \text{ if } p^{\nu}\parallel q \text{ and } ap^e\equiv b \text{ mod } q/p^{\nu},\\
& 0 \ \ \ \ \ \ \ \ \ \ \ \text{ otherwise}.
\end{aligned}\right.
\end{equation}
Note that the condition $ap^e\equiv b \text{ mod } q$ implies that $p\nmid q$ since $(b,q)=1$. Therefore, choosing $y=(q\log q)^2$ in (5.4), and adding the contribution of the principal character (which was evaluated in Lemma 5.2) we obtain from (5.5) that the RHS of (5.4) equals
\begin{equation*}
\begin{aligned}
&- \phi(q)\sum_{\substack{n\geq 1\\ bn\equiv a \text{ mod }q}}\frac{\Lambda(n)}{n}e^{-\frac{n}{y}} - \phi(q)\sum_{\substack{n\geq 1\\ an\equiv b \text{ mod }q }}\frac{\Lambda(n)}{n}e^{-\frac{n}{y}}
-\sum_{p^{\nu}\parallel q}\phi\left(\frac{q}{p^{\nu}}\right)\sum_{\substack{e\geq 1\\ap^e\equiv b  \text{ mod } q/p^{\nu}}}\frac{\log p}{p^e}e^{-\frac{p^e}{y}}\\
&-\sum_{p^{\nu}\parallel q}\phi\left(\frac{q}{p^{\nu}}\right)\sum_{\substack{e\geq 1\\bp^e\equiv a  \text{ mod } q/p^{\nu}}}\frac{\log p}{p^e}e^{-\frac{p^e}{y}}+4\log q+ O(\log\log q).
\end{aligned}
\end{equation*}
Moreover, if $n\geq 2y\log y$, then $e^{-n/y}\leq 1/n$. This implies that
$$ \sum_{\substack{n\geq 2y\log y\\ bn\equiv a \text{ mod }q}}\frac{\Lambda(n)}{n}e^{-\frac{n}{y}}+\sum_{p^{\nu}\parallel q}\sum_{\substack{e\geq 2\log y\\bp^e\equiv a  \text{ mod } q/p^{\nu}}}\frac{\log p}{p^e}e^{-\frac{p^e}{y}} \ll \sum_{n\geq 2y\log y}\frac{\Lambda(n)}{n^2}\ll \frac{1}{q^2}.$$
Notice that when $p^{\nu}\parallel q$ we have $\phi(q/p^{\nu})= \phi(q)/(p^{\nu-1}(p-1))$ since $(p^{\nu}, q/p^{\nu})=1$. Thus, using that $1-e^{-t}\leq 2t $ for all $t>0$, we obtain
$$ \sum_{p^{\nu}\parallel q}\phi\left(\frac{q}{p^{\nu}}\right)\sum_{\substack{1\leq e\leq 2\log y\\bp^e\equiv a  \text{ mod } q/p^{\nu}}}\frac{\log p}{p^e}\left(1-e^{-\frac{p^e}{y}}\right)\ll \frac{1}{q\log q}\sum_{p| q}\frac{\log p}{p-1}\leq \frac{1}{q}.$$
The proposition follows upon collecting the above estimates.
\end{proof}
 Next, we establish the following lemma which, when combined with Proposition 5.1, yields $B_q(a,b)\ll \phi(q)$.
\begin{lem} Let $q$ be a large positive integer, $(a,b)\in \mathcal{A}_2(q)$, and denote by $s$ the least positive residue of $ab^{-1}\text{ mod } q$. Put $x=(q\log q)^2$. Then
$$  \sum_{\substack{n\leq 2x\log x\\ bn\equiv a \text{ mod }q}}\frac{\Lambda(n)}{n}e^{-n/x}=\frac{\Lambda(s)}{s}+O\left(\frac{\log^2 q}{q}\right).$$

\end{lem}
\begin{proof} Since $1-e^{-t}\leq 2t$ for all $t>0$, then
$$ \frac{\Lambda(s)}{s}e^{-s/x}=\frac{\Lambda(s)}{s}+ O\left(\frac{1}{q^{2}}\right).$$
On the other hand if $n\neq s$ is a positive integer such that $n\equiv s\text{ mod }q$, then $n=s+jq$ for some $j\geq 1$. Therefore, we have
$$ \sum_{\substack{n\leq 2x\log x\\ n\equiv s \text{ mod }q}}\frac{\Lambda(n)}{n}e^{-n/x}- \frac{\Lambda(s)}{s}\ll \log q \sum_{1\leq j\leq q\log^4q}\frac{1}{s+jq}+ \frac{1}{q^2}\ll \frac{\log^2 q}{q}.$$
\end{proof}
\begin{cor} For any $(a,b)\in \mathcal{A}_2(q)$ we have
$$|B_q(a,b)|\ll\phi(q).$$
\end{cor}
\begin{proof} First we note that $\Lambda(s)/s\leq (\log s)/s$ which is a decreasing function for $s\geq 3$. Moreover, the term $\Lambda(q/(q,a-b))/\phi(q/(q,a-b))$ is non-zero only when $q/(q,a-b)=p^{l}$ for some prime $p\geq 2$ and $l\geq 1$. In this case
$$\frac{\Lambda(q/(q,a-b))}{\phi(q/(q,a-b))}= \frac{\log p}{p^{l-1}(p-1)}\leq \frac{\log p}{p-1}\leq  \log 2.$$
Finally we have
$$ \sum_{p^{\nu}\parallel q}\sum_{\substack{1\leq e\leq 2\log x\\bp^e\equiv a  \text{ mod } q/p^{\nu}}}\frac{\log p}{p^{e+\nu-1}(p-1)}\leq \sum_{p|q}\frac{\log p}{(p-1)^2}\ll 1.$$
 Thus by Lemma 5.3 and Proposition 5.1, the result follows.
\end{proof}
In the remaining part of this section, we prove Theorems 5 and 7.
\begin{proof}[Proof of Theorem 5]
Surprisingly the lower bound is much easier to establish than the upper bound. Indeed we only use the definition of $B_q(a,b)$ in this case.

{\bf The lower bound}.
Note that
$$ \sum_{(a,b)\in\mathcal{A}_2(q)}B_q(a,b)=\sum_{\substack{\chi\neq \chi_0 \\ \chi\text{ mod } q}}\sum_{\gamma_{\chi}>0}\frac{1}{\frac14 +\gamma_{\chi}^2}\sum_{\substack{a\text{ mod }q\\ (a,q)=1}}\sum_{\substack{b\neq a\text{ mod }q\\ (b,q)=1}}\left(\chi(a/b)+\chi(b/a)\right).$$
Write $s\equiv ab^{-1}\text{ mod } q$. When $a$ is fixed and $b$ varies over all reduced residue classes distinct from $a$, $s$ runs over all reduced residue classes different from $1$. Then, using the orthogonality relations for characters we obtain
$$ \sum_{\substack{a\text{ mod }q\\ (a,q)=1}}\sum_{\substack{b\neq a\text{ mod }q\\ (b,q)=1}}\left(\chi(a/b)+\chi(b/a)\right)= -2\phi(q).$$
Therefore, since $|\mathcal{A}_2(q)|=\phi(q)^2-\phi(q),$ and $N_q=\phi(q)(\log q+O(\log\log q))$ we deduce that
$$ \frac{1}{|\mathcal{A}_2(q)|} \sum_{(a,b)\in\mathcal{A}_2(q)}|B_q(a,b)|\geq
 -\frac{1}{|\mathcal{A}_2(q)|}\sum_{(a,b)\in\mathcal{A}_2(q)}B_q(a,b)= \log q +O(\log\log q).$$

{\bf The upper bound}.
We use Proposition 5.1. First, remark that
$ \sum_{(a,b)\in\mathcal{A}_2(q)}l_q(a,b)\leq \phi(q)$, which implies that the contribution of this sum to the upper bound in Theorem 5 is $\ll 1$. Take $1\leq a, b\leq q-1$.
Let $d=(q,a-b)$ and write $a-b=ds$. Then $-q/d\leq s\leq q/d$ and $(s,q/d)=1$. On the other hand, for any choice of $d$ and $s$ satisfying these conditions there are at most $\phi(q)$ pairs $(a,b)$ such that $1\leq a\neq b\leq q-1$, $a$ and $b$ are coprime to $q$ and $a-b=ds$. Thus we obtain
\begin{equation}
\begin{aligned}
\sum_{(a,b)\in \mathcal{A}_2(q)}\frac{\Lambda\left(\frac{q}{(q,a-b)}\right)}{\phi\left(\frac{q}{(q,a-b)}\right)}&\leq \phi(q)\sum_{d|q}\frac{\Lambda(q/d)}{\phi(q/d)}\sum_{\substack{-q/d\leq s\leq q/d\\ (s,q/d)=1}}1\\
&= 2\phi(q)\sum_{d|q}\Lambda(q/d)=2\phi(q)\log q.
\end{aligned}
\end{equation}
Let $x=(q\log q)^2$. Then
\begin{equation*}
\begin{aligned}
 \sum_{(a,b)\in \mathcal{A}_2(q)}\sum_{\substack{n\leq 2x\log x\\ n\equiv ab^{-1} \text{ mod }q}}\frac{\Lambda(n)}{n}e^{-n/x}&=\sum_{\substack{n\leq 2x\log x\\(n,q)=1}}\frac{\Lambda(n)}{n}e^{-n/x} \sum_{\substack{(a,b)\in \mathcal{A}_2(q)\\ab^{-1}\equiv n \text{ mod }q}}1\\
 &\leq \phi(q) \sum_{\substack{n\leq 2x\log x\\(n,q)=1}}\frac{\Lambda(n)}{n}e^{-n/x}\\
 &\leq  2\phi(q)\log q + O\left(\phi(q)\log\log q\right),
\end{aligned}
\end{equation*}
which follows from Lemma 5.2.
Finally, using an analogous argument we deduce that
$$ \sum_{(a,b)\in \mathcal{A}_2(q)}\sum_{p^{\nu}\parallel q}\sum_{\substack{1\leq e\leq 2\log x\\ap^e\equiv b  \text{ mod } q/p^{\nu}}}\frac{\log p}{p^{e+\nu-1}(p-1)}\leq \phi(q)\sum_{p|q}\sum_{e=1}^{\infty}\frac{\log p}{p^e(p-1)}\ll \phi(q),$$
which completes the proof.
\end{proof}

\begin{proof}[Proof of Theorem 7] First, notice that $|\mathcal{A}_r(q)|=\phi(q)^r+O_r\left(\phi(q)^{r-1}\right).$ Let $S_q$ be the set of pairs  $(a,b)\in\mathcal{A}_2(q)$ such that $|B_q(a,b)|\geq \sqrt{\phi(q)}$. Then Theorem 5 shows that
$$ |S_q| \sqrt{\phi(q)}\leq \sum_{(a,b)\in \mathcal{A}_2(q)}|B_q(a,b)|\ll\phi(q)^2\log q,$$
which gives $|S_q|\ll \phi(q)^{3/2}\log q.$ Now define $\Omega_r(q)$ to be the set of $r$-tuples $(a_1,\dots,a_r)\in \mathcal{A}_r(q)$ such that $(a_i,a_j)\in S_q$ for some $1\leq i\neq j\leq r$. Then $|\Omega_r(q)|\ll_r \phi(q)^{r-1/2}\log q$. On the other hand, if $(a_1,\dots,a_r)\in \mathcal{A}_r(q)\setminus \Omega_r(q)$ then  $|B_q(a_i,a_j)|\leq \sqrt{\phi(q)}$ for all $1\leq i<j\leq r$. Hence, in this case, we infer from Theorem 1 that
$$ \delta_{q;a_1,\dots,a_r}= \frac{1}{r!}-\frac{1}{\sqrt{N_q}}\sum_{1\leq j\leq r}\alpha_j(r)C_q(a_j) +O\left(\frac{1}{\sqrt{N_q\log q}}\right).$$
Since the $C_q(a_j)$ are integers, the theorem follows upon noting that
$$\sum_{1\leq j\leq r}\alpha_j(r)C_q(a_j)\neq 0\implies |\sum_{1\leq j\leq r}\alpha_j(r)C_q(a_j)|\gg_r 1.$$

\end{proof}

\section{Extreme values of $B_q(a,b)$ and explicit constructions}

Throughout this section we take the residues $a_i$ modulo $q$ so that $|a_i|\leq q/2.$ The proofs of Theorems 2, 4 and 6 are based on explicit constructions of the $a_i$. Our strategy consists in choosing these residue classes in such a way to make exactly one of the terms $B_q(a_i,a_j)$ large (using Proposition 6.1 below) and all the others small. Moreover, since this term must be  negative (see remark 5.1), we use Lemma 6.3 below to control the sign of its contribution to the asymptotic formula of the densities $\delta_{q;a_1,\dots,a_r}$. When $|a|$ and $|b|$ are relatively small compared to $q$, we can precisely understand in which cases does $B_q(a,b)$ get large. Let us define the real valued function
\begin{equation} \Lambda_0(x):=\begin{cases} \displaystyle{\frac{\Lambda(x)}{x}} & \text{ if } x\in \mathbb{N},\\
0 & \text{ otherwise.}\\
\end{cases}
\end{equation}
\begin{pro} Let $q$ be a large integer and $a,b$ be distinct integers coprime to $q$ such that $1\leq |a|,|b|<q/2$.

I) If $a$ and $b$ have different signs, then
$$ B_q(a,b)= -\phi(q)l(a,b)\log 2+ O\left((|a|+|b|)\log^2q\right),$$
where $l(a,b)=1$ if $a=-b$, and equals $0$ otherwise.

II) If $a$ and $b$ have the same sign, then
$$B_q(a,b)= -\phi(q)\Lambda_0\left(\frac{\max(|a|,|b|)}{\min(|a|,|b|)}\right)+ O\left((|a|+|b|)\log^2q\right) .$$
\end{pro}
An important ingredient to the proof of this result is the following lemma.
\begin{lem}Let $q$ be a large integer and $a,b$ be distinct integers coprime to $q$ such that $1\leq |a|, |b|< q/2$. Then

$$ \sum_{p^{\nu}||q}\sum_{\substack{ 1\leq e\leq 5\log q\\ ap^e\equiv b \text{ mod } q/p^{\nu}}} \frac{\log p}{p^{e+\nu-1}(p-1)}\ll \frac{(|a|+|b|)\log^2 q}{q}.$$
\end{lem}
\begin{proof} [Proof] First note that $ap^{e}-b$ can not vanish since $p|q$ and $(ab,q)=1$. This implies that when $q/p^{\nu}$ divides $ap^{e}-b$, we must have $q/p^{\nu}\leq |a|p^{e}+|b|$, so that $p^{e+\nu}\geq q/(|a|+|b|)$.
 Therefore the sum we are seeking to bound is
$$ \ll \frac{(|a|+|b|)\log q}{q}\sum_{p|q}\frac{p\log p}{p-1}\ll
\frac{(|a|+|b|)\log^2q}{q}.$$
\end{proof}
\begin{proof}[Proof of Proposition 6.1]
The proof relies on Proposition 5.1. Since $|a|,|b|< q/2$ then $a+b\equiv 0\text{ mod }q$ implies that $a=-b$. Moreover, notice that $(q,a-b)\leq |a|+|b|$, which gives
$$ \frac{\Lambda\left(\frac{q}{(q,a-b)}\right)}
{\phi\left(\frac{q}{(q,a-b)}\right)}\ll \frac{(|a|+|b|)\log^2q}{q},$$
using the standard estimate $\phi(q)\gg q/\log q$.
Combining this bound with Proposition 5.1 and Lemmas 5.3 and 6.2 we obtain
\begin{equation}
 B_q(a,b)= -\phi(q)\left(l(a,b)\log 2+\frac{\Lambda(s_1)}{s_1}+\frac{\Lambda(s_2)}{s_2}\right)+O\left((|a|+|b|)\log^2 q\right),
\end{equation}
where $s_1$ and $s_2$ denote the least positive residues of $ba^{-1}$ and $ab^{-1}$ modulo $q$, respectively.

Let us first prove part I. Since $a$ and $b$ have different signs, then $s_1a\neq b$ and $s_2b\neq a$. On the other hand we have that $q$ divides both $s_1a-b$ and $s_2b-a$. This implies that $q\leq s_i(|a|+|b|)$ for $i=1,2$, and thus $s_i\geq q/(|a|+|b|).$ Hence we get
$$ \frac{\Lambda(s_1)}{s_1}+\frac{\Lambda(s_2)}{s_2}\ll \frac{(|a|+|b|)\log q}{q},$$
which, in view of equation (6.2), gives the first part of the Proposition.

Now, if $a$ and $b$ have the same sign, then $l(a,b)=0$, and $|a|\neq |b|.$ Without any loss of generality we may assume that $|a|<|b|$. Then $s_2b\neq a$, which as before implies that $\Lambda(s_2)/s_2\ll (|a|+|b|)(\log q)/q.$ Furthermore, if $a|b$ then $s_1=|b|/|a|$; while if $a\nmid b$ then $s_1\geq q/(|a|+|b|)$, and thus $\Lambda(s_1)/s_1\ll (|a|+|b|)(\log q)/q$ in this case. Therefore, we obtain
$$ \frac{\Lambda(s_1)}{s_1}=\Lambda_0\left(\frac{|b|}{|a|}\right)
+O\left(\frac{(|a|+|b|)\log q}{q}\right).$$
Hence, part II follows upon combining these estimates with equation (6.2).
\end{proof}

Our next result determines the signs of some of the integrals $\beta_{j,k}(r)$.
\begin{lem} For $r\geq 3$ we have
$ \beta_{1,r}(r)<0$ and $\beta_{r-1,r}(r)>0.$
\end{lem}
\begin{proof} First we have
\begin{equation*}
\begin{aligned}
&(2\pi)^{r/2}\beta_{r-1,r}(r)= \int_{x_1>x_2>\cdots>x_r}x_{r-1}x_r\exp\left(-\frac{x_1^2+\cdots+x_r^2}{2}\right)dx_1\dots dx_r\\
&= \int_{x_1>\cdots>x_{r-2}}\exp\left(-\frac{x_1^2+\cdots+x_{r-2}^2}{2}\right)
\int_{-\infty}^{x_{r-2}}x_{r-1}e^{-\frac{x_{r-1}^2}{2}}\int_{-\infty}^{x_{r-1}}x_{r}e^{-\frac{x_{r}^2}{2}}dx_rdx_{r-1}\cdots dx_1\\
&= -\int_{x_1>\cdots>x_{r-2}}\exp\left(-\frac{x_1^2+\cdots+x_{r-2}^2}{2}\right)
\int_{-\infty}^{x_{r-2}}x_{r-1}e^{-x_{r-1}^2}dx_{r-1}\cdots dx_1\\
&= \frac{1}{2}\int_{x_1>\cdots>x_{r-2}}\exp\left(-\frac{x_1^2+\cdots+x_{r-3}^2+ 3x_{r-2}^2}{2}\right)dx_{r-2}\cdots dx_1 >0.\\
\end{aligned}
\end{equation*}
Similarly we get
\begin{equation*}
\begin{aligned}
&(2\pi)^{r/2}\beta_{1,r}(r)= \int_{x_1>x_2>\cdots>x_r}x_1x_r\exp\left(-\frac{x_1^2+\cdots+x_r^2}{2}\right)dx_1\dots dx_r\\
&= \int_{x_2>\cdots>x_{r-1}}\exp\left(-\frac{x_2^2+\cdots+x_{r-1}^2}{2}\right)
\int_{x_{2}}^{\infty}x_1e^{-\frac{x_1^2}{2}}\int_{-\infty}^{x_{r-1}}x_{r}e^{-\frac{x_{r}^2}{2}}dx_1dx_{r}dx_{r-1}\cdots dx_2\\
&= -\int_{x_1>\cdots>x_{r-2}}\exp\left(-\frac{2x_2^2+x_3^2\cdots+x_{r-2}^2+2x_{r-1}^2}{2}\right)dx_{r-1}\cdots dx_2<0.\\
\end{aligned}
\end{equation*}
\end{proof}
 Before proving  Theorems 2, 4 and 6, let us first define some notation. Let $q$ be a large positive integer. Define $p$ to be the largest prime divisor of $q$, and denote by $p_0$ the least non-quadratic residue modulo $p$ (if $p=2$ take $p_0=3$). Then Burgess's bound on short character sums (see \cite{IK}) implies that $p_0\leq p^{1/(4\sqrt{e})+\epsilon}\leq q^{1/4}$. Moreover note that $p_0$ is a prime and is also a non-square modulo $q$. Furthermore we shall denote by $p_1<p_2$  the smallest prime numbers such that $p_i\neq p_0$ for $i=1,2$, and $(p_1p_2,q)=1$. Then one has $p_1<p_2\leq 2\log q$, in view of the fact that $\prod_{p\leq z}p=e^{z+o(z)}$ which follows from the prime number theorem.

\begin{proof}[Proof of Theorem 2] The first part that $|\delta_{q; a_1,\dots,a_r}-1/r!|\ll_r 1/\log q$ follows from combining Theorem 1  with Corollary 5.4 and the fact that $|C_q|=q^{o(1)}.$ Concerning the second part we first take $a_1=1$, $a_r=-1$ and $a_j= (p_1p_2)^{2j}$ for $2\leq j\leq r-1$. Then $|a_j|\leq (2\log q)^{4(r-1)}$ for all $1\leq j\leq r$. Using part II of Proposition 6.2 we obtain
$$ B_q(a_j,a_k)\ll (\log q)^{4r}, \text{ for all } 1\leq j<k\leq r-1,$$
since $p_1p_2|a_k/a_j$ in this case. Furthermore, part I of the same proposition implies that
$$ B_q(a_j,a_r)\ll (\log q)^{4r} \text{ for all } 2\leq j\leq r-1,$$
and $$B_q(a_1,a_r)= -\phi(q)\log 2 + O\left(\log^{2}q\right).$$ Therefore by Theorem 1 and Lemma 6.3 we deduce that
$$ \delta_{q;a_1,\dots,a_r}= \frac{1}{r!}+\frac{\beta_{1,r}(r)B_q(a_1,a_r)}{N_q}+ O_{\epsilon}\left(\frac{1}{\phi(q)^{1/2-\epsilon}}\right)>\frac{1}{r!}+
\frac{|\beta_{1,r}(r)|\log 2}{2\log q}.$$
Furthermore taking $b_1=a_{r-1}$, $b_{r-1}=a_1$ and $b_j=a_j$ for all other values of $j$, we obtain by Lemma 6.3 that
$$ \delta_{q;b_1,\dots,b_r}= \frac{1}{r!}+\frac{\beta_{r-1,r}(r)B_q(b_{r-1},b_r)}{N_q}+ O_{\epsilon}\left(\frac{1}{\phi(q)^{1/2-\epsilon}}\right)<\frac{1}{r!}
-\frac{|\beta_{r-1,r}(r)|\log 2}{2\log q},$$
completing the proof.

\end{proof}
\begin{proof}[Proof of Theorem 4] We only need to construct the squares $a_j$ modulo $q$, since in this case $\delta_{q; ba_1,\dots, ba_r}= \delta_{q; a_1,\dots,a_r}$ for any residue class $b$ modulo $q$ by Theorem 2 of Feuerverger and Martin \cite{FeM}. Thus it suffices to take $b_j=ba_j$ for any non-square $b$ modulo $q$, to get the analogous result for non-squares.

Let $a_1=1,a_r=p_1^2$ and $a_j=(p_1p_2)^{2j}$ for $2\leq j\leq r-1$. Then $a_j\leq (2\log q)^{4(r-1)}$ for all $1\leq j\leq r$. Moreover for $1\leq j<k\leq r-1$ notice that $p_1p_2|a_k/a_j$. Therefore part II of Proposition 6.1 gives that
$$
 B_q(a_j,a_k)\ll (\log q)^{4r}, \text{ for all } 1\leq j<k\leq r-1.
$$
 and
$$ B_q(a_j,a_r)\ll  (\log q)^{4r}, \text{ for all } 2\leq j\leq r-1,$$
since $p_1p_2|a_j/a_r$ in this case.
Finally, since $a_r/a_1=p_1^2$, we have
$$ B_q(a_1,a_r)= -\phi(q)\frac{\log p_1}{p_1^{2}} + O\left((\log q)^{4r}\right).$$
Thus combining  these estimates with Corollary 3 and Lemma 6.3 we deduce that
\begin{equation} \delta_{q;a_1,\dots, a_r}=\frac{1}{r!}+\frac{\beta_{1,r}(r)B_q(a_1,a_r)}{N_q}+O\left(\frac{(\log q)^{4r} }{\phi(q)}\right)>\frac{1}{r!}+ \frac{|\beta_{1,r}(r)|}{5\log^3 q}.
\end{equation}
 Furthermore, let $\sigma$ be the permutation on the set $\{1,\dots,r\}$ defined by $\sigma(1)=r-1$, $\sigma(r-1)=1$, and $\sigma(j)=j$ for all other values of $j$. Then using Lemma 6.3 we obtain similarly to (6.3) that
\begin{equation} \delta_{q;a_{\sigma(1)},\dots, a_{\sigma(r)}}=\frac{1}{r!}+\frac{\beta_{r-1,r}(r)B_q(1,p_1^2)}{N_q}
+O\left(\frac{(\log q)^{4r} }{\phi(q)}\right)<\frac{1}{r!}- \frac{|\beta_{r-1,r}(r)|}{5\log^3 q},
\end{equation}
which completes the proof.

\end{proof}

\begin{proof}[Proof of Theorem 6] The main idea of the proof relies on the fact (proved in part II of Proposition 6.1) that when $a,b>0$ and $a,b$ are small comparatively to $q$, the quantity $B_q(a,b)$ is small unless $\max(a,b)/\min(a,b)$ equals a prime power.  Since $(\kappa_1,\dots,\kappa_r)\neq (0,\dots,0)$ then $\kappa_l\neq 0$ for some $1\leq l\leq r$.
 \smallskip

\noindent {\bf Case 1}: $\kappa_r\neq 0$ or $\kappa_1\neq 0$.

 We only handle the case $\kappa_r\neq 0$, since the treatment of the case $\kappa_1\neq 0$ follows simply by switching $a_1$ with $a_r$, and $b_1$ with $b_r$ in every construction we make below. Assume first that $\kappa_r>0$. In this case take $a_1=1$, $a_j= p_0(p_1p_2)^{2j}$ for $2\leq j\leq r-1$ and
 $a_r= (p_1p_2)^2$. Then $a_1$ and $a_r$ are squares and $a_j$ is a non-square modulo $q$ for all $2\leq j\leq r-1$. Moreover choose $b_j=a_j$ for all $1\leq j\leq r-1$ and $b_r=p_0$. In this case $b_1$ is the only square among the $b_j$ modulo $q$.
 Since $C_q(1)>-1$ we get that
$$ \sum_{j=1}^r\kappa_jC_q(a_j)-\sum_{j=1}^r\kappa_jC_q(b_j)= \kappa_rC_q(a_r)-\kappa_rC_q(b_r)= \kappa_r(C_q(1)+1)>0.$$
On the other direction, note that $|a_j|\leq q^{1/4}(2\log q)^{4(r-1)}$ for all $1\leq j\leq r$, and that $p_1p_2$ divides $\max(a_j,a_k)/\min(a_j,a_k)$ for all $1\leq j<k\leq r$. Therefore, upon using part II of Proposition 6.1 we deduce that
$$ |B_q(a_j,a_k)|\ll q^{1/4}(\log q)^{4r}\text { for all } 1\leq j<k\leq r.$$
Hence by Theorem 1 we obtain
\begin{equation}
 \delta_{q;a_1,\dots,a_r}=\frac{1}{r!}+ O_{\epsilon}\left(\frac{1}{\phi(q)^{1/2-\epsilon}}\right).
\end{equation}
Similarly, part II of Proposition 6.1 gives that $|B_q(b_j,b_k)|\ll q^{1/4}(\log q)^{4r}$ for all $\{j,k\}\neq \{1,r\}$ and
$$ B_q(b_1,b_r)= -\phi(q)\frac{\log p_0}{p_0}+ O\left(q^{1/4}(\log q)^{4r}\right).$$
Thus using Theorem 1 along with Lemma 6.3 and equation (6.5) we get
$$  \delta_{q;b_1,\dots,b_r}=\frac{1}{r!}+\frac{\beta_{1,r}(r)B_q(b_1,b_r)}{N_q} + O_{\epsilon}\left(\frac{1}{\phi(q)^{1/2-\epsilon}}\right)>\frac{1}{r!}+ \frac{|\beta_{1,r}(r)|\log p_0}{2p_0\log q}> \delta_{q;a_1,\dots,a_r}.$$

Now suppose that $\kappa_r<0$. In this case we choose $a_1=1$ and $a_j= p_0(p_1p_2)^{2j}$ for all $2\leq j\leq r$ (so that $a_1$ is the only square among the $a_j$); and $b_j=a_j$ for all $1\leq j\leq r-1$,  and $b_r=p_1^2$ (in this case both $b_1$ and $b_r$ are squares modulo $q$). Then similarly to the case $k_r>0$, one has
$$ \sum_{j=1}^r\kappa_jC_q(a_j)-\sum_{j=1}^r\kappa_jC_q(b_j)= -\kappa_r(1+C_q(1))>0,$$
$$\delta_{q;a_1,\dots,a_r}= \frac{1}{r!}+ O_{\epsilon}\left(\frac{1}{\phi(q)^{1/2-\epsilon}}\right),$$
and
$$ \delta_{q;b_1,\dots,b_r}=\frac{1}{r!}+\frac{\beta_{1,r}(r)B_q(b_{1},b_r)}{N_q}+ O_{\epsilon}\left(\frac{1}{\phi(q)^{1/2-\epsilon}}\right)>\frac{1}{r!}+
\frac{|\beta_{1,r}(r)|\log p_1}{2p_1^2\log q}>\delta_{q;a_1,\dots,a_r},$$
using Theorem 1, part II of Proposition 6.1 and Lemma 6.3.

\smallskip

\noindent {\bf Case 2}: $\kappa_l\neq 0$ for some $2\leq l\leq r-1$.

As before assume first that $k_l>0$. For the $a_i$ we choose $a_1=1, a_l=(p_1p_2)^2,$ and $a_j=p_0(p_1p_2)^{4j}$ for $2\leq j\neq l\leq r$; and for the $b_i$ we take $b_l= p_0(p_1p_2)^{4l}$, $b_r=p_0$ and $b_j=a_j$ for all other values of $j$. Then, an analogous argument to Case 1 gives that
$$ \sum_{j=1}^r\kappa_jC_q(a_j)-\sum_{j=1}^r\kappa_jC_q(b_j)= \kappa_l(C_q(1)+1)>0, \text{ and } \delta_{q; b_1,\dots,b_r}>\delta_{q; a_1,\dots,a_r},$$
if $q$ is large. Finally if $\kappa_l<0$, we choose $a_1=1$, $a_r= (p_1p_2)^4$ and $a_j= p_0(p_1p_2)^{4j}$ for $2\leq j \leq r-1$; and $b_l= (p_1p_2)^4$, $b_r=p_1^2$ and $b_j=a_j$ for all other values for $j$, to deduce the desired conclusion.
\end{proof}
\section{$q$-extremely biased races}

The idea behind the proof of Theorem 3 is to observe that when the $a_i$ are small comparatively to $q$, the term $B_q(a_i,a_j)$ have a large contribution to the density $\delta_{q; a_1,\dots,a_r}$ if and only if $a_i=-a_j$ or $a_i$ and $a_j$ have the same sign and $\max(|a_i|,|a_j|)/\min(|a_i|,|a_j|)$ equals a prime power (this is proved in Proposition 6.1). The first step is to reduce to the case $r=3$ (which is easier to deal with) using the following lemma.
\begin{lem} Let $r\geq 3$ be a fixed integer, $q$ be a large positive integer and $(a_1,\dots,a_r)\in \mathcal{A}_r(q)$. If there exist $1\leq i_1<i_2<i_3\leq r$ such that the race $\{q;a_{i_1},a_{i_2},a_{i_3}\}$ is $q$-extremely biased, then the race $\{q; a_1,\dots, a_r\}$ is $q$-extremely biased.
\end{lem}
\begin{proof}
Suppose that there exist $1\leq i_1<i_2<i_3\leq r$ with the property that the race $\{q;a_{i_1},a_{i_2},a_{i_3}\}$ is $q$-extremely biased. Then, for some permutation $\nu$ of the set $\{i_1,i_2,i_3\}$ we have $|\delta_{q; a_{\nu(i_1)}, a_{\nu(i_2)}, a_{\nu(i_3)}}-1/6|\gg 1/\log q$. Let $j_l= \nu(i_l)$,  and define $S$ to be set of all permutations $\sigma$ of $\{1,\dots,r\}$ such that $\sigma(j_1)>\sigma(j_2)>\sigma(j_3).$ Then, using the definition of the densities $\delta_{q;a_1,\dots,a_r}$, we have
\begin{equation}
\delta_{q;a_{j_1},a_{j_2}, a_{j_3}}= \sum_{\sigma\in S}\delta_{q;a_{\sigma(1)},\dots, a_{\sigma(r)}}.
\end{equation}
Now, a simple combinatorial argument shows that $|S|=r!/3!$. Hence we obtain from (7.1) that
$$ \frac{1}{\log q}\ll \left|\delta_{q;a_{j_1},a_{j_2}, a_{j_3}}-\frac{1}{6}\right|\leq \sum_{\sigma\in S}\left|\delta_{q;a_{\sigma(1)},\dots, a_{\sigma(r)}}-\frac{1}{r!}\right|\ll_r \max_{\sigma\in S} \left|\delta_{q;a_{\sigma(1)},\dots, a_{\sigma(r)}}-\frac{1}{r!}\right|,$$
which implies that the race $\{q; a_1,\dots,a_r\}$ is $q$-extremely biased.
\end{proof}

The next step is to investigate the main contribution to $B_q(a,b)$ when $a,b>0$ are relatively small compared to $q$ and $\max(a,b)/\min(a,b)$ equals a prime power. To this end we establish some properties of the function $\Lambda_0(x)$ defined in (6.1).
\begin{lem} The maximum of $\Lambda_0(x)$ over $\mathbb{R}$ equals $(\log 3)/3$. Moreover, if $n$ is a positive integer with $\Lambda_0(n)\neq 0$, then $\Lambda_0(m)=\Lambda_0(n)$ implies that $m=n$.
\end{lem}
\begin{proof} We know that $\Lambda_0(x)\neq 0$ if and only if  $x=p^l$ for some prime $p$ and a positive integer $l$. In this case $\Lambda_0(x)=(\log p)/p^l\leq \Lambda_0(p).$ The first part follows upon noting that the function $(\log x)/x$ is decreasing for $x\geq 3$ and $(\log 3)/3>(\log 2)/2.$

If $\Lambda_0(m)=\Lambda_0(n)\neq 0$, then there exist primes $p_1,p_2$ and positive integers $e_1,e_2$ such that $n=p_1^{e_1}$, $m=p_2^{e_2}$ and
$(\log p_1)/p_1^{e_1}=(\log p_2)/p_2^{e_2}.$ This implies that $p_1^{p_2^{e_2}}= p_2^{p_1^{e_1}}$, from which one can deduce that $p_1=p_2$ and thus $e_1=e_2$.
\end{proof}
\begin{lem} Let $a_1$, $a_2$ and $a_3$ be distinct positive real numbers. Define
$$X_1= \frac{\max(a_1,a_2)}{\min(a_1,a_2)}, X_2=\frac{\max(a_2,a_3)}{\min(a_2,a_3)}, \text{ and } X_3=\frac{\max(a_1,a_3)}{\min(a_1,a_3)}.$$
If one of the values $\Lambda_0(X_1)$, $\Lambda_0(X_2)$ and $\Lambda_0(X_3)$ is non-zero, then there exists a permutation $\sigma$ of the set $\{1,2,3\}$ such that
$$ \Lambda_0(X_{\sigma(1)})+ \Lambda_0(X_{\sigma(2)})- 2\Lambda_0(X_{\sigma(3)})\neq 0.$$

\end{lem}
\begin{proof} Assume without loss of generality that $a_1<a_2<a_3$. In this case $X_1= a_2/a_1$, $X_2=a_3/a_2$ and $X_3=a_3/a_1.$ Suppose that for all permutations $\sigma$ of the set $\{1,2,3\}$ we have  $ \Lambda_0(X_{\sigma(1)})+ \Lambda_0(X_{\sigma(2)})- 2\Lambda_0(X_{\sigma(3)})= 0.$ Then we must have $\Lambda_0(X_1)=\Lambda_0(X_2)=\Lambda_0(X_3)$. Furthermore since this value is non-zero we get by Lemma 7.2 that $X_1=X_2=X_3$. However this can not hold since $X_3\neq X_1$ by our hypothesis on the $a_i$.
\end{proof}
\begin{proof}[Proof of Theorem 3]
Assume First that neither i) nor ii) hold. In this case Proposition 6.1 implies that $B_q(a_j,a_k)=O_A\left(\log^2q\right),$ for all $1\leq j<k\leq r$. Inserting this estimate in Corollary 3 in the case where the $a_i$ are all squares (or all non-squares) modulo $q$, gives  $|\delta_{q;a_1\dots,a_r}-1/r!|\ll_{A,r}(\log q)/q$. Now if this is not the case then Theorem 1 implies that
$|\delta_{q;a_1\dots,a_r}-1/r!|\ll_{\epsilon,r} q^{-1/2+\epsilon}.$ Thus in both cases the race $\{q;a_1,\dots,a_r\}$ is not $q$-extremely biased.

Next, let us consider the case where $a_j=-a_k=a$ for some $1\leq j<k\leq r$. Since $r\geq 3$, then there exists $b\in \{a_1,\dots,a_r\}$ such that $b\neq a$ and $b\neq -a$. By Lemma 7.1 it suffices to prove that the race $\{q;a,-a,b\}$ is $q$-extremely biased. Without any loss of generality we may assume that $a$ and $b$ have the same sign (otherwise simply switch $a$ and $-a$). Applying Proposition 6.1 we obtain $B_q(a,-a)=-\phi(q)\log 2+ O_A(\log^2 q)$, $B_q(b,-a)=O_A(\log^2q)$ (since $b$ and $-a$ have different signs and $b-a\neq 0$) and
$$ B_q(a,b)= -\phi(q)\Lambda_0\left(\frac{\max(|a|,|b|)}{\min(|a|,|b|)}\right)+O_A(\log^2 q)\geq -\frac{\log 3}{3}\phi(q)+ O_A(\log^2 q),$$
which follows from Lemma 7.2. Inserting these estimates in Corollary 2, and recalling that $N_q\sim \phi(q)\log q$ and $|C_q(a)|=q^{o(1)}$, we get
$$ \delta_{q;a, b,-a}\geq \frac{1}{6} +\frac{2\log 2-(\log 3)/3}{8\pi\sqrt{3}}\frac{1}{\log q},$$
if $q$ is large enough, so that the race $\{q;a,-a,b\}$ is $q$-extremely biased.

Now, suppose that $a_i\neq -a_j$ for all $1\leq i<j\leq r$, and that there exist $b_1,b_2\in \{a_1,\dots,a_r\}$ such that $b_1=p^kb_2$ for some prime $p$, and a positive integer $k$. In this case part II of Proposition 6.1 yields
\begin{equation}
 B_q(b_1,b_2)= -\phi(q)\frac{\log p}{p^k}+O_A(\log^2q).
\end{equation}
Since $r\geq 3$, then there exists $b_3\in \{a_1,\dots,a_r\}$ with $b_3\neq b_i$ for $i=1,2$. First if $b_3$ and $b_1$ have different signs, then part I of Proposition 6.1 implies that
\begin{equation}
B_q(b_1,b_3), B_q(b_2,b_3)\ll_A\log^2 q.
\end{equation} Therefore, inserting the estimates (7.2) and (7.3) in Corollary 2 gives
$$ \delta_{q; b_1,b_2,b_3}=\frac{1}{6}-\frac{1}{4\pi\sqrt{3}}\frac{\log p}{p^k}(1+o(1)),$$
and thus the race $\{q;b_1,b_2,b_3\}$ is $q$-extremely biased. Hence, it only remains to handle the case where all the $b_i$ have the same sign. Let us denote by $S_3$ the set of all permutations of $\{1,2,3\}$. Since $|b_1|,|b_2|$ and $|b_3|$ are distinct by our hypothesis, and $\Lambda_0(|b_1|/|b_2|)\neq 0$, then Lemma 7.3 shows that there exists $\sigma\in S_3$ such that
$$ \Lambda_0(X_{\sigma(1)})+ \Lambda_0(X_{\sigma(2)})- 2\Lambda_0(X_{\sigma(3)})\neq 0,$$
where
$$X_1=\frac{|b_1|}{|b_2|}=p^k, X_2=\frac{\max(|b_2|,|b_3|)}{\min(|b_2|,|b_3|)}, \text{ and } X_3=\frac{\max(|b_1|,|b_3|)}{\min(|b_1|,|b_3|)}.$$
Therefore, upon using part II of Proposition 6.1 along with Corollary 2, we deduce that
$$ \max_{\nu\in S_3}\left|\delta_{q; b_{\nu(1)}, b_{\nu(2)}, b_{\nu(3)}}-\frac{1}{6}\right|\gg \frac{|\Lambda_0(X_{\sigma(1)})+ \Lambda_0(X_{\sigma(2)})- 2\Lambda_0(X_{\sigma(3)})|}{\log q},$$
which implies that the race $\{q;b_1,b_2,b_3\}$ is $q$-extremely biased. Thus, appealing to Lemma 7.1 the result follows.

\end{proof}
\section{Another proof for the asymptotic in two-way races}
In this section we derive Fiorilli and Martin \cite{FiM} asymptotic formula for the densities in the case $r=2$, using a slight modification of the method used to establish Theorem 1. In the version presented below, our main concern is to obtain the main term of (2.2) without giving much attention to the error term, in order to keep the exposition simple. Nonetheless, our approach would give an asymptotic expansion for $\delta_{q;a_1,a_2}$ with little extra work, if one allows more terms in the asymptotic series of the Fourier transform $\hat{\mu}_{q;a_1,a_2}$ in Lemma 8.1 below. Indeed we shall establish that

\begin{equation}
\delta_{q;a_1,a_2}= \frac{1}{2}- \frac{C_q(a_1)-C_q(a_2)}{\sqrt{2\pi V_q(a_1,a_2)}}+ O\left(\frac{C_q(1)^2\log^2q}{ V_q(a_1,a_2)}\right),
\end{equation}
for $(a_1,a_2)\in \mathcal{A}_2(q)$.
We begin by proving the analogue of Proposition 3.3
\begin{lem} For $t=(t_1,t_2)\in \mathbb{R}^2$ with $||t||\leq N_q^{1/4}$ we have
$$
\hat{\mu}_{q;a_1,a_2}
\left(\frac{t_1}{\sqrt{N_q}},\frac{t_2}{\sqrt{N_q}}\right)
=\exp\left(-\frac{t_1^2+t_2^2}{2}-\frac{B_q(a_1,a_2)}{N_q}t_1t_2\right)
F_{q;a_1,a_2}(t_1,t_2),$$
where
$$ F_{q;a_1,a_2}(t_1,t_2)=1+\frac{i}{\sqrt{N_q}}(C_q(a_1)t_1+C_q(a_2)t_2)
+ O\left(\frac{||t||^4}{N_q}+ \frac{||t||^2C_q(1)^2}{N_q}\right).$$
\end{lem}
\begin{proof} We follow closely the proof of Proposition 3.3. Indeed, for $||t||\leq N_q^{1/4}$ the explicit formula (2.1) implies that $\log\hat{\mu}_{q;a_1,a_2}
\left(t_1N_q^{-1/2},t_2N_q^{-1/2}\right)$ equals
\begin{equation*}
\begin{aligned}
&\frac{i}{\sqrt{N_q}}(C_q(a_1)t_1+C_q(a_2)t_2)- \frac{1}{N_q}\sum_{\substack{\chi\neq\chi_0\\ \chi \text{ mod } q}}\sum_{\gamma_{\chi}>0}\frac{|\chi(a_1)t_1+\chi(a_2)t_2|^2}
{\frac14+\gamma_{\chi}^2}
+O\left(\frac{||t||^4}{N_q}\right)\\
=&\frac{i}{\sqrt{N_q}}(C_q(a_1)t_1+C_q(a_2)t_2)
-\frac{t_1^2+t_2^2}{2}-\frac{B_q(a_1,a_2)}{N_q}
t_1t_2+O\left(\frac{||t||^4}{N_q}\right).
\end{aligned}
\end{equation*}
Thus, the lemma follows upon noting that
$$\exp\left(\frac{i}{\sqrt{N_q}}(C_q(a_1)t_1+C_q(a_2)t_2)\right)= 1+\frac{i}{\sqrt{N_q}}(C_q(a_1)t_1+C_q(a_2)t_2)+ O\left(\frac{||t||^2C_q(1)^2}{N_q}\right).$$
\end{proof}
Our next result is an analogue of Lemma 4.2 in the case of a bivariate normal distribution.
\begin{lem} Let $\rho$ be a real number such that $|\rho|\leq 1/2$, $n_1$, $n_2$ are fixed non-negative integers, and $M$ a large positive number. Then
\begin{equation*}
\begin{aligned}
&\int_{||t||\leq M}e^{i(t_1x_1+t_2x_2)}t_1^{n_1}t_2^{n_2}\exp\left(
-\frac{t_1^2+t_2^2+2\rho t_1t_2}{2}\right)dt_1dt_2\\
&= \frac{1}{i^{n_1+n_2}}\frac{\partial^{n_1+n_2}
\Phi_{\rho}(x_1,x_2)}{\partial x_1^{n_1}\partial x_2^{n_2}}      + O\left(\exp\left(-\frac{M^2}{8}\right)\right),
\end{aligned}
\end{equation*}
where
$$\Phi_{\rho}(x_1,x_2)= \frac{2\pi}{\sqrt{1-\rho^2}}
\exp\left(-\frac{1}{2(1-\rho^2)}(x_1^2+x_2^2-2\rho x_1x_2)\right).
$$
\end{lem}
\begin{proof} First, notice that $t_1^2+t_2^2+2\rho t_1t_2\geq (t_1^2+t_2^2)/2 $ which follows from the fact that $|t_1t_2|\leq (t_1^2+t_2^2)/2$. This implies that the integral we are seeking to estimate equals
$$ \int_{t\in \mathbb{R}^2}e^{i(t_1x_1+t_2x_2)}t_1^{n_1}t_2^{n_2}\exp\left(
-\frac{t_1^2+t_2^2+2\rho t_1t_2}{2}\right)dt_1dt_2 + O\left(\exp\left(-\frac{M^2}{8}\right)\right).
$$
Moreover, since the last integral is absolutely and uniformly convergent for $(x_1,x_2)\in \mathbb{R}^2$, we get that
$$  \int_{t\in \mathbb{R}^2}e^{i(t_1x_1+t_2x_2)}t_1^{n_1}t_2^{n_2}\exp\left(
-\frac{t_1^2+t_2^2+2\rho t_1t_2}{2}\right)dt_1dt_2= \frac{1}{i^{n_1+n_2}}\frac{\partial^{n_1+n_2}
\Phi_{\rho}(x_1,x_2)}{\partial x_1^{n_1}\partial x_2^{n_2}},
$$
where
$$ \Phi_{\rho}(x_1,x_2) = \int_{t\in \mathbb{R}^2}e^{i(t_1x_1+t_2x_2)}\exp\left(
-\frac{t_1^2+t_2^2+2\rho t_1t_2}{2}\right)dt_1dt_2.
$$
On the other hand, remark that $\frac{\sqrt{1-\rho^2}}{2\pi}\Phi_{\rho}(x_1,x_2)$ is the characteristic function of the bivariate normal distribution whose density is
$$ f(t_1,t_2)=\frac{\sqrt{1-\rho^2}}{2\pi}\exp\left(
-\frac{t_1^2+t_2^2+2\rho t_1t_2}{2}\right).$$
Therefore, we obtain that
$$ \frac{\sqrt{1-\rho^2}}{2\pi}\Phi_{\rho}(x_1,x_2)= \exp\left(-\frac{1}{2(1-\rho^2)}(x_1^2+x_2^2-2\rho x_1x_2)\right),$$
which completes the proof.

\end{proof}

We are now ready to establish (8.1).
We begin by following the proof of Theorem 1. Write $\mu_{q}=\mu_{q;a_1,a_2}$ and let $R=\sqrt{N_q}\log q$. Then Proposition 4.1 yields
$$ \delta_{q;a_1,a_2}=\int_{-R<y_2<y_1<R}d\mu_q(y_1,y_2)+ O\left(\exp\left(-\frac{\log^2q}{10}\right)\right).$$
Applying the Fourier inversion formula to the measure $\mu_q$ gives that
\begin{equation}
 \delta_{q;a_1,a_2}= \frac{1}{(2\pi)^2}\int_{-R<y_2<y_1<R}\int_{s\in \mathbb{R}^2}e^{i(s_1y_1+s_2y_2)}\hat{\mu}_q(s_1,s_2)d{\bf s}d{\bf y}+ O\left(\exp\left(-\frac{\log^2q}{10}\right)\right).
\end{equation}
Moreover, using Proposition 3.2 with $\epsilon=\log q N_q^{-1/2}$ gives
$$\int_{s\in \mathbb{R}^2}e^{i(s_1y_1+s_2y_2)}\hat{\mu}_q(s_1,s_2)d{\bf s}= \int_{||s||\leq \epsilon }e^{i(s_1y_1+s_2y_2)}\hat{\mu}_q(s_1,s_2)d{\bf s} +O\left(\exp\left(-c\log^2q\right)\right),
$$
for some constant $c>0$.
Inserting this estimate in (8.2), and making the change of variables $t_j=\sqrt{N_q}s_j$ and $x_j=y_j/\sqrt{N_q}$ for $j=1,2$, we infer from Lemma 8.1 that
\begin{equation}
\begin{aligned}
\delta_{q;a_1,a_2}&= \frac{1}{(2\pi)^2}
\int_{-\log q<x_2<x_1<\log q}\int_{||t||<\log q}e^{i(t_1x_1+t_2x_2)}\hat{\mu}_q\left(\frac{t_1}{\sqrt{N_q}},
\frac{t_2}{\sqrt{N_q}}\right)d{\bf t}d{\bf x} \\
&+O\left(\exp\left(-\log^{3/2}q\right)\right).\\
&= I_0+ \frac{iC_q(a_1)}{\sqrt{N_q}}I_1+
\frac{iC_q(a_2)}{\sqrt{N_q}}I_2 +O\left(\frac{C_q(1)^2\log^2q}{N_q}\right),
\end{aligned}
\end{equation}
where
$$ I_0= \frac{1}{(2\pi)^2}
\int_{-\log q<x_2<x_1<\log q}\int_{||t||<\log q}e^{i(t_1x_1+t_2x_2)}
\exp\left(-\frac{t_1^2+t_2^2}{2}-\frac{B_q(a_1,a_2)}{N_q}t_1t_2\right)
d{\bf t}d{\bf x},$$
and
$$ I_j= \frac{1}{(2\pi)^2}
\int_{-\log q<x_2<x_1<\log q}\int_{||t||<\log q}e^{i(t_1x_1+t_2x_2)}
t_j\exp\left(-\frac{t_1^2+t_2^2}{2}-\frac{B_q(a_1,a_2)}{N_q}t_1t_2\right)
d{\bf t}d{\bf x},$$
for $j=1,2$. We shall first evaluate $I_0$. Let $\rho= B_q(a_1,a_2)/N_q$. Then corollary 5.4 implies that $|\rho|\leq 1/2$ for $q$ large. Hence Lemma 8.2 yields
\begin{equation*}
\begin{aligned}
I_0&= \frac{1}{2\pi\sqrt{1-\rho^2}}
\int_{-\log q<x_2<x_1<\log q}\exp\left(-\frac{1}{2(1-\rho^2)}(x_1^2+x_2^2-2\rho x_1x_2)\right)dx_1dx_2\\
&+ O\left(\exp\left(-\frac{\log^2 q}{10}\right)\right).\\
\end{aligned}
\end{equation*}
Now the integral on the RHS of the last estimate equals
$$\frac{1}{2\pi\sqrt{1-\rho^2}}
\int_{x_1>x_2}\exp\left(-\frac{1}{2(1-\rho^2)}(x_1^2+x_2^2-2\rho x_1x_2)\right)dx_1dx_2 + O\left(\exp\left(-\frac{\log^2 q}{10}\right)\right).$$
Therefore, using that the integrand is symmetric in $x_1$ and $x_2$, along with the fact that
$$ \frac{1}{2\pi\sqrt{1-\rho^2}}
\int_{-\infty}^{\infty}\int_{-\infty}^{\infty}\exp\left(-\frac{1}{2(1-\rho^2)}(x_1^2+x_2^2-2\rho x_1x_2)\right)dx_1dx_2= 1, $$ we deduce that
\begin{equation}
 I_0= \frac{1}{2}+  O\left(\exp\left(-\frac{\log^2 q}{10}\right)\right).
\end{equation}
Using similar ideas along with Lemma 8.2 gives
\begin{equation*}
\begin{aligned}
 I_1&= \frac{1}{(2\pi)^2 i}\int_{x_1>x_2}\frac{\partial \Phi_{\rho}(x_1,x_2)}{\partial x_1}dx_1dx_2+O\left(\exp\left(-\frac{\log^2 q}{10}\right)\right)\\
 &=-\frac{1}{(2\pi)^2 i}\int_{-\infty}^{\infty} \Phi_{\rho}(x_2,x_2)dx_2+O\left(\exp\left(-\frac{\log^2 q}{10}\right)\right) .\\
 \end{aligned}
 \end{equation*}
 Furthermore, one has
 $$ \int_{-\infty}^{\infty} \Phi_{\rho}(y,y)dy= \frac{2\pi}{\sqrt{1-\rho^2}}\int_{-\infty}^{\infty} \exp\left(-\frac{y^2}{2}\left(\frac{2}{1+\rho}\right)\right)dy= \frac{2\pi^{3/2}}{\sqrt{1-\rho}}.$$
 Note that $2(1-\rho)= V_q(a_1,a_2)/N_q$. Thus, upon combining the above estimates we get
\begin{equation}
I_1= - \frac{\sqrt{N_q}}{i\sqrt{2\pi V_q(a_1,a_2)}} + O\left(\exp\left(-\frac{\log^2 q}{10}\right)\right).
\end{equation}
Similarly one obtains
\begin{equation}
I_2= \frac{\sqrt{N_q}}{i\sqrt{2\pi V_q(a_1,a_2)}} + O\left(\exp\left(-\frac{\log^2 q}{10}\right)\right).
\end{equation}
Finally, inserting the estimates (8.4)-(8.6) into equation (8.3), and using the fact that $V_q(a_1,a_2)\sim 2N_q$ give the desired result.

\end{document}